\documentclass[a4paper, 11pt]{article}

\usepackage[utf8]{inputenc}
\usepackage[T1]{fontenc}
\usepackage{geometry}
\usepackage[francais]{layout}
\usepackage{amsthm}
\usepackage{amsmath}
\usepackage{array,booktabs}
\usepackage{multirow}
\usepackage{array}
\usepackage{comment}
\usepackage{listings}
\usepackage{verbatim}
\usepackage[nottoc, notlof, notlot]{tocbibind}
\usepackage{textcomp}
\usepackage{lmodern}
\usepackage{amssymb}
\usepackage{fancyvrb}
\usepackage{graphicx}
\usepackage{enumerate}
\usepackage{varioref}
\usepackage[dvipsnames]{xcolor}
\usepackage{caption}
\usepackage{wrapfig}
\usepackage{tikz}
\usepackage{hyperref}
\hypersetup{
colorlinks=true,
linkcolor=RoyalBlue,
urlcolor=PineGreen,
citecolor=OliveGreen,
}
\usepackage[capitalise]{cleveref}
\usetikzlibrary{shapes.multipart}
\usetikzlibrary{positioning}
\usetikzlibrary{arrows}
\usepackage{subcaption}
\usepackage{comment}
\usepackage{epsfig}
\usepackage[
backend=biber,
style=alphabetic,
citestyle=authoryear,firstinits=true,abbreviate=true,isbn = false, url = false,doi=false,eprint=false, backend=bibtex
]{biblatex}
\renewbibmacro{in:}{}
\AtEveryBibitem{%
  \clearfield{month}%
}
\usepackage{comment}

\newcommand{\bl}[1]{{\color{Black}{#1}}}

    \usepackage{etoolbox}

    \makeatletter
    \patchcmd{\lsthk@SelectCharTable}{%
      \lst@ifbreaklines\lst@Def{`)}{\lst@breakProcessOther)}\fi}{}{}{}

    \makeatother
    \makeatletter
    \newcounter {subsubsubsection}[subsubsection]
    \renewcommand\thesubsubsubsection{\thesubsubsection .\@alph\c@subsubsubsection}
    \newcommand\subsubsubsection{\@startsection{subsubsubsection}{4}{\z@}%
                                         {-3.25ex\@plus -1ex \@minus -.2ex}%
                                         {1.5ex \@plus .2ex}%
                                         {\normalfont\normalsize\bfseries}}
    \newcommand\l@subsubsubsection{\@dottedtocline{3}{10.0em}{4.1em}}
    \newcommand{\subsubsubsectionmark}[1]{}
    \makeatother
\newcommand\R{\mathbb{R}}

\newcommand{\e}{{\mathrm e}}

\newcounter{exo}
\newcommand{\vertiii}[1]{{\left\vert\kern-0.25ex\left\vert\kern-0.25ex\left\vert #1 
    \right\vert\kern-0.25ex\right\vert\kern-0.25ex\right\vert}}
\renewcommand{\textbf}[1]{\begingroup\bfseries\mathversion{bold}#1\endgroup}
\newtheorem{prop}{Proposition}[section]
\labelformat{prop}{Proposition~#1}

\labelformat{formalprop}{Formal Proposition~#1}

\newtheorem{theorem}{Theorem}[section]

\newtheorem{lem}{Lemma}[section]

\newtheorem{rem}{Remark}[section]

\labelformat{cor}{Corollary~#1}
\labelformat{fig}{Figure~#1}

\begin{document}
\title{Dynamics of dirac concentrations in the evolution of quantitative alleles with sexual reproduction}
\author{L. Dekens\footnote{Institut Camille Jordan, UMR5208 UCBL/CNRS, Universit\'e de Lyon, 69100 Villeurbanne, France. Email: \href{mailto:dekens@math.univ-lyon1.fr}{dekens@math.univ-lyon1.fr}.} and S. Mirrahimi\footnote{{\color{Black}Institut Montpelliérain Alexander Grothendieck, Univ Montpellier, CNRS, Montpellier, France. Email: \href{mailto:sepideh.mirrahimi@umontpellier.fr}{sepideh.mirrahimi@umontpellier.fr}.}}}
\date{\today}

\maketitle
\begin{abstract}
{\color{black}A proper understanding of the links between varying gene expression levels and complex trait adaptation is still lacking}, despite recent {\color{black}advances} in sequencing techniques leading to new insights on their importance in some evolutionary processes. This calls for extensions of the continuum-of-alleles framework first introduced by \textcite{Kimura731} that bypass the classical Gaussian approximation. Here, we propose a novel mathematical framework to study the evolutionary dynamics of quantitative alleles for sexually reproducing populations under natural selection and competition through an integro-differential equation. It involves a new reproduction operator {\color{black}which is nonlinear and nonlocal. This reproduction operator is different from the infinitesimal operator used in other studies with sexual reproduction because of different underlying genetic structures}. In an asymptotic regime where {\color{black}initially the population has a small phenotypic variance}, we analyse the long-term dynamics of the {\color{black}phenotypic distributions} according to the methodology of small variance (\cite{Diekmann_Jabin_Mischler_Perthame_2005}). {\color{black}In particular, we prove that the reproduction operator strains the limit distribution to be a product measure.} {\color{black} Under some assumptions on the limit equation, we show that the population remains monomorphic, that is the phenotypic distribution remains concentrated as a moving Dirac mass. Moreover, in the case of a monomorphic distribution, we derive a canonical equation describing the dynamics of the dominant alleles.}
\end{abstract}

\tableofcontents

\section{Introduction}

\subsection{Model and biological motivations.}

The development and popularization of sequencing techniques of the last twenty years has been leading to a greater understanding of regulatory mechanisms of gene expression levels and to new insights on their importance in evolutionary trajectories of complex traits (see the recent theory of degeneration of the Y chromosome \textcite{Lenormand_Fyon_Sun_Roze_2020}).  However, a complete picture of the relationship between varying gene expression levels and phenotypic adaptation is yet to be drawn (\cite{Romero_Ruvinsky_Gilad_2012}). To model varying gene expression levels on a trait under selection, one has to think of the effects of a gene as quantitative rather than discrete. One class of models that was motivated by a similar perspective stems from the reference study \textcite{Kimura731}: the continuum-of-alleles models in quantitative genetics, that assume that mutations produce always slightly new allelic effects, so that the allelic effect space is considered as continuous. The method indicated by \textcite{Kimura731} is adapted for asexual populations, or haploid sexual populations with only one locus contributing to the trait under quadratic stabilizing selection. Under these specific assumptions, \textcite{Kimura731} shows that the allelic effects are normally distributed under mutation-selection balance. Several studies (\cite{Latter_1970,Lande_1975}) extended the model to account for finite {\color{black}number of loci} with additive effects on the trait for sexual reproducing populations, still relying on the essential link between quadratic stabilizing selection and multivariate normal allelic distributions to derive quantitative information from their non-linear model. The aim of this paper is therefore to first study a quantitative genetics model that can account for polygenic traits under general {\color{black}selection functions} (not restricted to quadratic {\color{black}and considering situations where the alleles do not necessarily have additive effects}), in a sexually reproducing population regulated by competition for resources. 
More precisely, we are interested in the following integro-differential equation, where $\boldsymbol{t}\geq0$ denotes the time:

\begin{equation}
    \label{P_n_ante}
\tag{$P\left(\boldsymbol{n}\right)$}
    \begin{aligned}
    \begin{cases}
    \partial_{\boldsymbol{t}} \boldsymbol{n}(\boldsymbol{t},x,y) = \frac{r}{2}\,\left[\frac{\boldsymbol{\rho^Y}(\boldsymbol{t},x)\,\boldsymbol{\rho^X }(\boldsymbol{t},y)}{\boldsymbol{\rho}(\boldsymbol{t})}+\boldsymbol{n}(\boldsymbol{t},x,y)\right] - \left( m(x,y)+\kappa \, \boldsymbol{\rho}(\boldsymbol{t})\right)\,\boldsymbol{n}(\boldsymbol{t},x,y),\\
    \\
    \boldsymbol{\rho^X}(\boldsymbol{t},y) = \displaystyle\int_I \boldsymbol{n}(\boldsymbol{t},x',y)\,dx',\quad \boldsymbol{\rho^Y}(\boldsymbol{t},x) = \displaystyle\int_J \boldsymbol{n}(\boldsymbol{t},x,y')\,dy', \quad \boldsymbol{\rho}(t) = \displaystyle\int_{I\times J} \boldsymbol{n}(\boldsymbol{t},x',y')\,dx'\,dy',\\
    \\
    \boldsymbol{n}(0,x,y) = \boldsymbol{n^0}(x,y).
    \end{cases}
    \end{aligned}
\end{equation}

Here, $\boldsymbol{n}(t,x,y)$ denotes the allelic density of individuals of a haploid sexually reproducing population carrying the {\color{black}quantitative} alleles $x$ and $y$ at two unlinked loci of interest. {\color{black}The alleles $x$ and $y$ are taken in compact allelic spaces $I$ and $J$}. Individuals experience mortality by natural selection at a rate $m(x,y)\in C^1(I\times J)$ depending on their genotype $(x,y) \in I\times J$ and regulated by a uniform competition for resources with intensity $\kappa$. The first term in the r.h.s of \ref{P_n_ante} is the reproduction term, which translates how alleles are transmitted across generations under random mating \bl{at rate $r$}. According to Mendel's laws, \bl{there are two equiprobable configurations which lead to an offspring being born with $x$ and $y$ alleles. In the first configuration, each allele comes from a different parent, and the complementary alleles of both parents can be chosen arbitrarily, which results in the non-linear term involving the marginal contributions of each parent $\frac{\boldsymbol{\rho^Y}(\boldsymbol{t},x)\,\boldsymbol{\rho^X }(\boldsymbol{t},y)}{\boldsymbol{\rho}(\boldsymbol{t})}$. In the second configuration, both alleles come from the same parent and the other parent can be chosen arbitrarily in the population, which results in the simpler term $\boldsymbol{n}(\boldsymbol{t},x,y)$}. 

\begin{rem}[One-locus diploid population.]
\label{rem:diploid}
One can notice that up to setting $\tilde{m}:=m-\frac{r}{2}$, $\tilde{r} = \frac{r}{2}$, \ref{P_n_ante} also describes the dynamics of a population of diploid individuals (\bl{each individual has two alleles at each locus}) whose adaptation is determined by the two quantitative alleles $(x,y)$ carried at a single focal locus. The following equation was derived as deterministic limit of an individual-based model in \textcite{Collet_Méléard_Metz_2013}
\begin{equation}
    \label{P_n_ante_diploid}
\tag{$P_{\text{diploid}}\left(\boldsymbol{n}\right)$}
    \begin{aligned}
    \begin{cases}
    \partial_{\boldsymbol{t}} \boldsymbol{n}(\boldsymbol{t},x,y) = \tilde{r}\frac{\boldsymbol{\rho^Y}(\boldsymbol{t},x)\,\boldsymbol{\rho^X }(\boldsymbol{t},y)}{\boldsymbol{\rho}(\boldsymbol{t})} - \left( \tilde{m}(x,y)+\kappa \, \boldsymbol{\rho}(\boldsymbol{t})\right)\,\boldsymbol{n}(\boldsymbol{t},x,y),\\
    \\
    \boldsymbol{\rho^X}(\boldsymbol{t},y) = \displaystyle\int_I \boldsymbol{n}(\boldsymbol{t},x',y)\,dx',\quad \boldsymbol{\rho^Y}(\boldsymbol{t},x) = \displaystyle\int_J \boldsymbol{n}(\boldsymbol{t},x,y')\,dy', \quad \boldsymbol{\rho}(t) = \displaystyle\int_{I\times J} \boldsymbol{n}(\boldsymbol{t},x',y')\,dx'\,dy',\\
    \\
    \boldsymbol{n}(0,x,y) = \boldsymbol{n^0}(x,y).
    \end{cases}
    \end{aligned}
\end{equation} According to Mendel's laws, the copies $x$ and $y$ must be inherited {\color{black}each from a different parent} \bl{and the other copy of each parent can be chosen arbitrarily, which results in the same non-linear term as in the first configuration for the two-locus haploid case}. In the diploid case, $\tilde{r}$ is the reproduction rate and both the selection function $\tilde{m}$ and the initial genotypic density $\boldsymbol{n_0}$ are assumed symmetrical (requiring $I=J)$ (one can verify that the genotypic density $\boldsymbol{n}$ remains symmetrical at all times). {\color{black}All qualitative results will also be presented for this case in \cref{sec:formal}.}
\end{rem}

We place our analysis in an asymptotic regime where we consider that the initial distribution is concentrated, with a small variance $\varepsilon$ so that it is convenient to introduce the following transformation \bl{of the initial distribution}:
\begin{equation*}
    {\color{Black}\boldsymbol{n}^{0}} = \frac{e^{\frac{\boldsymbol{u^0_\varepsilon}}{\varepsilon}}}{\varepsilon}.
\end{equation*}
The motivation behind \bl{the latter} comes from a future project that will include mutations on the alleles with a small mutational variance of order {\color{black}$\varepsilon^2$}, \bl{which will allow the population to explore the allelic space beyond the support of the initial distribution (which it cannot do in the present model)}. Here, we expect that starting with an initial condition with such a small variance, the population density ${\color{Black}\boldsymbol{n}}$ solution of \ref{P_n_ante} would keep the same exponential form as above and would remain asymptotically concentrated with a small variance. Consequently, the dynamics of its mean, driven by natural selection with an intensity correlated to its variance, cannot be observed at shallow time scales, and \eqref{P_n_ante} needs to be adequately rescaled in order to explore long-term dynamics. To that effect, let us define the following rescaling in time:

\begin{equation*}
    t = \varepsilon \,\boldsymbol{t},\quad n_\varepsilon(t,\cdot,\cdot) = {\color{Black}\boldsymbol{n}}(\boldsymbol{t},\cdot,\cdot),\quad \rho^X_\varepsilon(t,\cdot) = {\color{Black}\boldsymbol{\rho^X}}(\boldsymbol{t},\cdot),\quad \rho^Y_\varepsilon(t,\cdot) = {\color{Black}\boldsymbol{\rho^Y}}(\boldsymbol{t},\cdot),\quad \rho_\varepsilon(t) = {\color{Black}\boldsymbol{\rho}}(\boldsymbol{t}).
\end{equation*}
Under the latter, the problem \eqref{P_n_ante} becomes, for $t\geq 0$, $(x,y) \in I\times J$:
\begin{equation}
\label{P_n}
\tag{$P\left(n_\varepsilon\right)$}
    \begin{aligned}
    \begin{cases}
    \varepsilon \,\partial_tn_\varepsilon(t,x,y) = \frac{r}{2}\,\left[\frac{\rho^Y_{\varepsilon}(t,x)\,\rho^X_{\varepsilon}(t,y)}{\rho_\varepsilon(t)} +n_\varepsilon(t,x,y)\right]- \left( m(x,y)+\kappa \, \rho_\varepsilon(t)\right)\,n_\varepsilon(t,x,y),\\
    \\
    \rho^X_\varepsilon(t,y) = \displaystyle\int_I n_\varepsilon(t,x',y)\,dx',\quad \rho^Y_\varepsilon(t,x) = \displaystyle\int_J n_\varepsilon(t,x,y')\,dy', \quad \rho_\varepsilon(t) = \displaystyle\int_{I\times J}\!\! n_\varepsilon(t,x',y')\,dx'\,dy',\\
    \\
    n_\varepsilon(0,x,y) = n^0_\varepsilon(x,y).
    \end{cases}
    \end{aligned}
\end{equation}
As we expect the density $n_\varepsilon$ to remain concentrated in our regime, the objective is to analytically describe the dynamics of the Dirac masses (ie. of the dominant alleles in the population), for various {\color{black}selection functions}.

\subsection{State of the art}

Integro-{\color{black}differential models} for quantitative genetics modelling the evolutionary dynamics of large sexually reproducing populations with selection have been on the rise recently, especially those that model the phenotypic trait inheritance according to the non-linear infinitesimal model introduced by \textcite{Fisher_1919} (\cite{Mirrahimi_Raoul_2013,Raoul_2017,Bourgeron_Calvez_Garnier_Lepoutre_2017,Calvez_Garnier_Patout_2019,Patout_2020,dekens_lavigne,dekens2020evolutionary,raoul2021exponential,dekens2021best}). According to the latter, the offspring's trait deviates from the mean parental trait according to a Gaussian kernel of fixed segregational variance. The classical interpretation is that the trait under consideration results from the combination of {\color{black} a large number of loci with small additive allelic effects} (\cite{Lange_1978,Bulmer_1980,Turelli_Barton_1994,Tufto_2000,Turelli_2017}), a framework rigorously justified in \textcite{Barton_Etheridge_Veber_2017}. In another study \textcite{perthame2021selectionmutation}, asymmetrical kernels are considered to model the effect of asymmetrical trait inheritance or fecundity on the asymptotic behaviour of the trait distribution. {\color{black}The present work also studies sexually reproducing populations, but the genetical framework  is different from the ones aforementioned: here, we consider that the allelic effects at the two loci are continuous and not necessarily small nor additive}.
\paragraph{Small variance methodology and long term-dynamics.}
We choose to place our study in the small variance methodology, introduced for quantitative genetics studies in \textcite{Diekmann_Jabin_Mischler_Perthame_2005} from a high-frequency method used in geometric optics. When the variance introduced by events of reproduction (by mutations, segregation...) is small compared to the reduction of diversity following natural selection, they propose to unfold Dirac singularities that are expected to arise using the so-called Hopf-Cole transform:

\begin{equation*}
    \boldsymbol{n}_{\varepsilon} = \frac{e^{\frac{\boldsymbol{u}_{\varepsilon}}{\varepsilon}}}{\varepsilon}.
\end{equation*}
The idea behind considering $u_\varepsilon$ instead of $n_\varepsilon$ stems from the fact that, when $\varepsilon$ vanishes, the limit $u$ (to be characterized) is expected to have more regularity than the (measure) limit $n$, making it more suitable for analysis. Moreover, $u$ would retain important quantitative information on the support of $n$. 

The small variance methodology has first been applied successfully to several quantitative genetics settings for \bl{asexually} reproducing populations in the regime of small variance of mutations:
adaptation to homogeneous environments \textcite{Perthame_Barles_2008,Barles_Mirrahimi_Perthame_2009}, to spatially heterogeneous environments \textcite{Mirrahimi_2017,Mirrahimi_Gandon_2020}, in a time-periodic environment \textcite{figueroa_mirrahimi}. Recently, it has been extended to quantitative genetics models for sexually reproducing populations characterized by complex traits inherited {\color{black}according to the aforementioned infinitesimal model} (\cite{Calvez_Garnier_Patout_2019,Patout_2020,dekens_lavigne,dekens2020evolutionary,dekens2021best}). However, the asymptotic analysis of this non-local, non-monotone, non-linear operator of reproduction presents great analytical challenges, and it has only been rigorously derived in a model for homogeneous environments (\cite{Calvez_Garnier_Patout_2019, Patout_2020}). {\color{black}The same methodology is used in \textcite{perthame2021selectionmutation} to study the asymptotic behaviour of the trait distribution under asymmetrical reproduction kernels. Here, as described above, our genetical framework differs significantly from the infinitesimal model's one. Therefore, it yields a different reproduction operator (see \cref{P_n}), which} is in fact {\color{Black}closer} to the ones used for asexual populations (\cite{Perthame_Barles_2008,Barles_Mirrahimi_Perthame_2009}), \bl{since integrating the reproduction term in \ref{P_n} with regard to $x$ or $y$ results in the same reproduction term as with clonal reproduction with a single trait and no mutations (.} However, here, the nonlinear nonlocal term describing the reproduction operator \bl{along with the fully general bivariate selection function $m$} still lead to new difficulties to be \bl{overcome}.

Let us then consider $\left(\boldsymbol{u^0_\varepsilon}\right)_{\varepsilon>0}$ a sequence in $C^1\left(I\times J\right)$, uniformly bounded when $\varepsilon$ vanishes. It defines subsequently a sequence of concentrated initial genotypic densities with decreasingly small variance (Hopf-Cole transform):

\begin{equation}
    \boldsymbol{n^0_\varepsilon} = \frac{e^{\frac{\boldsymbol{u^0_\varepsilon}}{\varepsilon}}}{\varepsilon}.
    \label{eq:IC_HC}
\end{equation}
Let us define $\boldsymbol{n}_\varepsilon$ \bl{the} solution of \eqref{P_n_ante} with initial distribution $\bl{\boldsymbol{n}^0_{\varepsilon}}$, and $\boldsymbol{u}_\varepsilon$ similarly as above:

\begin{equation*}
    \boldsymbol{n}_{\varepsilon} = \frac{e^{\frac{\boldsymbol{u}_{\varepsilon}}{\varepsilon}}}{\varepsilon}.
\end{equation*}
We expect indeed that starting with such an initial condition \eqref{eq:IC_HC}, the population density $\boldsymbol{n}_{\varepsilon}$ would keep the same exponential form and would remain asymptotically concentrated with a small variance. Consequently, the dynamics of its mean, driven by natural selection with an intensity correlated to its variance, cannot be observed at shallow time scales, and \eqref{P_n_ante} needs to be adequately rescaled in order to explore long term dynamics.

Moreover, in order to study the asymptotic properties of $n_\varepsilon$, we align with \textcite{Perthame_Barles_2008,Barles_Mirrahimi_Perthame_2009,Mirrahimi_2017,Mirrahimi_Gandon_2020}, and introduce the derived problem on $u_\varepsilon := \varepsilon \log\left(\varepsilon\,n_\varepsilon\right)$:
\begin{equation}
\begin{aligned}
\begin{cases}
\partial_t\,u_\varepsilon(t,x,y) = \bl{\frac{r}{2}\nu_\varepsilon(t,x,y)}- \left(m(x,y)+\kappa \,\rho_\varepsilon(t)-\frac{r}{2}\right),\\
u_\varepsilon(0,\cdot,\cdot) = u^0_\varepsilon,\\
\rho_\varepsilon = \displaystyle\iint_{I\times J} \frac{1}{\varepsilon} \exp\left[{\frac{u_\varepsilon(x',y')}{\varepsilon}}\right] dx'\,dy',
\end{cases}
\end{aligned}
\label{P_u}
\tag{$P_{u_\varepsilon}$}
\end{equation}
\bl{where
\[\nu_\varepsilon(t,x,y) := \frac{\rho^X_\varepsilon(t,y)\,\rho^Y_\varepsilon(t,x)}{n_\varepsilon(t,x,y)\rho_\varepsilon(t)}=\frac{1}{\rho_\varepsilon(t)} \displaystyle\iint_{I\times J} \frac{1}{\varepsilon}\,\exp\left[{\frac{u_\varepsilon(t,x,y')+u_\varepsilon(t,x',y)-u_\varepsilon(t,x,y)}{\varepsilon}}\right]\bl{dx'\,dy'}.\]}
\subsection{Assumptions}

We assume that the selection term $m$ satisfies the following \bl{regularity and technical bound}:
\begin{equation*}
\tag{H1}
{\color{black}m \in C^1(I\times J, \R_+), \quad  4\,\|m\|_\infty < r.}   
\label{hyp:m}
\end{equation*}
For $\varepsilon>0$, let $u^0_{\varepsilon}\in C^1(I\times J)$ be such that:
\begin{equation*}
\exists M>0,\quad \forall \varepsilon \leq 1,\quad
    \left\|u^0_{\varepsilon}\right\|_{W^{1,\infty}(I\times J)}\leq M.
    \tag{H2}
    \label{hyp:u0W1infty}
\end{equation*}
Then we define the initial state by
\begin{equation*}
    n^0_{\varepsilon} = \frac{e^{\frac{u^0_{\varepsilon}}{\varepsilon}}}{\varepsilon}.
\end{equation*}
Let us define the following uniform bounds:
\[\rho_0^- := \frac{r-\|m\|_\infty}{\kappa}, \quad \rho_0^+ := \frac{r}{\kappa}.\]
We assume that the initial size of population is bounded uniformly by $\rho_0^-$ and $\rho_0^+$:
\begin{equation*}
\tag{H3}
 \forall \varepsilon>0,\quad \rho_\varepsilon^0 := \iint_{I\times J} n^0_\varepsilon(x,y)\,dx\,dy \;\in \left]\rho_0^-,\rho_0^+\right[.
 \label{hyp:rho0unifbound}
\end{equation*}
Next, to prepare \ref{prop:nu}, we assume that there exists $0<\nu_m\leq 1-4\frac{\|m\|_\infty}{r}<1+4\frac{\|m\|_\infty}{r}\leq\nu_M$ such that:
\begin{equation*}
   \forall\varepsilon,\;\forall (x,y)\in I\times J,\quad \nu_m\leq \nu_\varepsilon^0(x,y) : = \frac{\rho_\varepsilon^{X,0}(y)\,\rho_\varepsilon^{Y,0}(x)}{n_\varepsilon^0(x,y)\,\rho_\varepsilon^0}\leq \nu_M.
    \tag{H4}
    \label{hyp:nu0}
\end{equation*}

\subsection{Presentation of the results and outline}

First, we show some preliminary results of well-posedness of \ref{P_n}:
\begin{theorem}
\label{thm:well_posedness}
Under the assumption \ref{hyp:rho0unifbound}, \eqref{P_n} has a unique solution with positive values $n_\varepsilon$ in \bl{$C^1(\R_+\times I\times J)$}. \bl{Moreover}, we have for all $\varepsilon$:
\[\forall t \in \R_+,\quad \rho_0^-\leq\rho_\varepsilon(t) \leq \rho_0^+.\]
Hence, for all $T>0$, $\left(n_\varepsilon\right)$ converges along subsequences in $L^\infty(w^*-[0,T],M(I\times J))$ toward a measure $n$ when $\varepsilon$ vanishes \bl{(where $M(I\times J)$ stands for the set of Radon measures equipped with the total variation norm)}.
\end{theorem}

We recall that we expect $n_\varepsilon$ to concentrate as $\varepsilon$ vanishes. As such, we expect the weak limit $n$ to be a sum of Dirac masses. The aim of this paper is to determine where $n$ is supported, that is to determine which alleles become dominant in the population. To study the asymptotic properties of $n$, it is more convenient to shift the asymptotic analysis from $n_\varepsilon$ on $u_\varepsilon = \varepsilon \log\left(\varepsilon\,n_\varepsilon\right)$. Consequently, the main result of this paper {\color{black}focuses} on the asymptotic behaviour of $u_\varepsilon$:

\begin{theorem}
Under the assumptions \ref{hyp:u0W1infty}-\ref{hyp:nu0}, for all $T>0$, $u_\varepsilon \underset{\varepsilon \to 0}{\longrightarrow} u$ in $C^0([0,T]\times I\times J)$ (along subsequences). Additionally, $u$ satisfies the following properties:
\begin{enumerate}
\item[(i)] $u$ is Lipschitz continuous,
    \item[(ii)] $u$ is non-positive and satisfies an additive separation of variables property:
{\color{black}\begin{equation}
\label{eq:u_additive}
    \forall (t,x,y) \in [0,T]\times I\times J,\quad u(t,x,y) = u^Y(t,x) + u^X(t,y):=\max u(t,x,\cdot) + \max u(t,\cdot,y).
\end{equation}}
Furthermore, we have at all time $t$: $\max u^Y(t,\cdot) = \max u^X(t,\cdot) = 0$.
\item[(iii)] $n(t,\cdot,\cdot)$ is supported at the zeros of $u(t,\cdot,\cdot)$ {\color{black}for} a.e. $t$:
    \begin{align*}
        \operatorname{supp}(n(t,\cdot,\cdot)) &\subset \{(x,y)\,|\,u(t,x,y) =0\}\\
        &=\{x\,|\,u^Y(t,x) =0\}\times\{y\,|\,u^X(t,y) =0\} .
    \end{align*}
    \item[(iv)] $u^X$ (resp. $u^Y)$ satisfies the following limit equation {\color{black}for} a.e. $y$:
    \begin{equation}
    \label{eq:limit_equation}
    \forall t \in [0,T]\quad u^X(t,y) = u^X(0,y) + r\,t-\kappa\,\int_0^t\rho(s)\,ds - \int_0^t\left\langle \phi^X(t,\cdot,y),m(\cdot,y)\right\rangle\,ds,
\end{equation}
    where \bl{$\rho = \langle n,\mathbf{1}_{I\times J}\rangle\in L^\infty([0,T])$} and $\phi^X$ is the limit of $\frac{n_\varepsilon}{\rho_\varepsilon^X}$ in $L^\infty\left(w^*-[0,T]\times I, M(I)\right)$. Moreover, {\color{black}for a.e. (t,y)
    \begin{equation*}
        \operatorname{supp}\left(\phi^X(t,\cdot,y)\right) =\{x\,|\,u^Y(t,x) =0\}.
    \end{equation*}}
\end{enumerate}
\label{thm:convergence_u}
\end{theorem}

The second and third point of the results in \cref{thm:convergence_u} highlight the originality of this problem: the limit $u$ separates the variables additively and therefore, the limit measure $n$ is a product measure. This asymptotic decorrelation of the effects of the two loci relies on the following proposition, that is key to establish the convergence stated in \cref{thm:convergence_u}:
\begin{prop}
Let us assume \ref{hyp:nu0}. \bl{For} all $T>0$, let $n_\varepsilon$ be the positive solution of \eqref{P_n} on $[0,T]$. Then\bl{, the following holds}:
\[\forall t\in [0,T],\quad \forall (x,y) \in I\times J,\quad 0<\nu_m\leq \nu_\varepsilon(t,x,y)\leq \nu_M.\]
\label{prop:nu}
\end{prop}
Indeed, the compactness result of \ref{prop:nu} together with some a priori estimates relying on a maximum principle  yield the convergence of \cref{thm:convergence_u} thanks to the Arzela-Ascoli theorem (see \bl{\cref{fig:layout_convergence}} for a flowchart that exposes the layout of the different results). 
 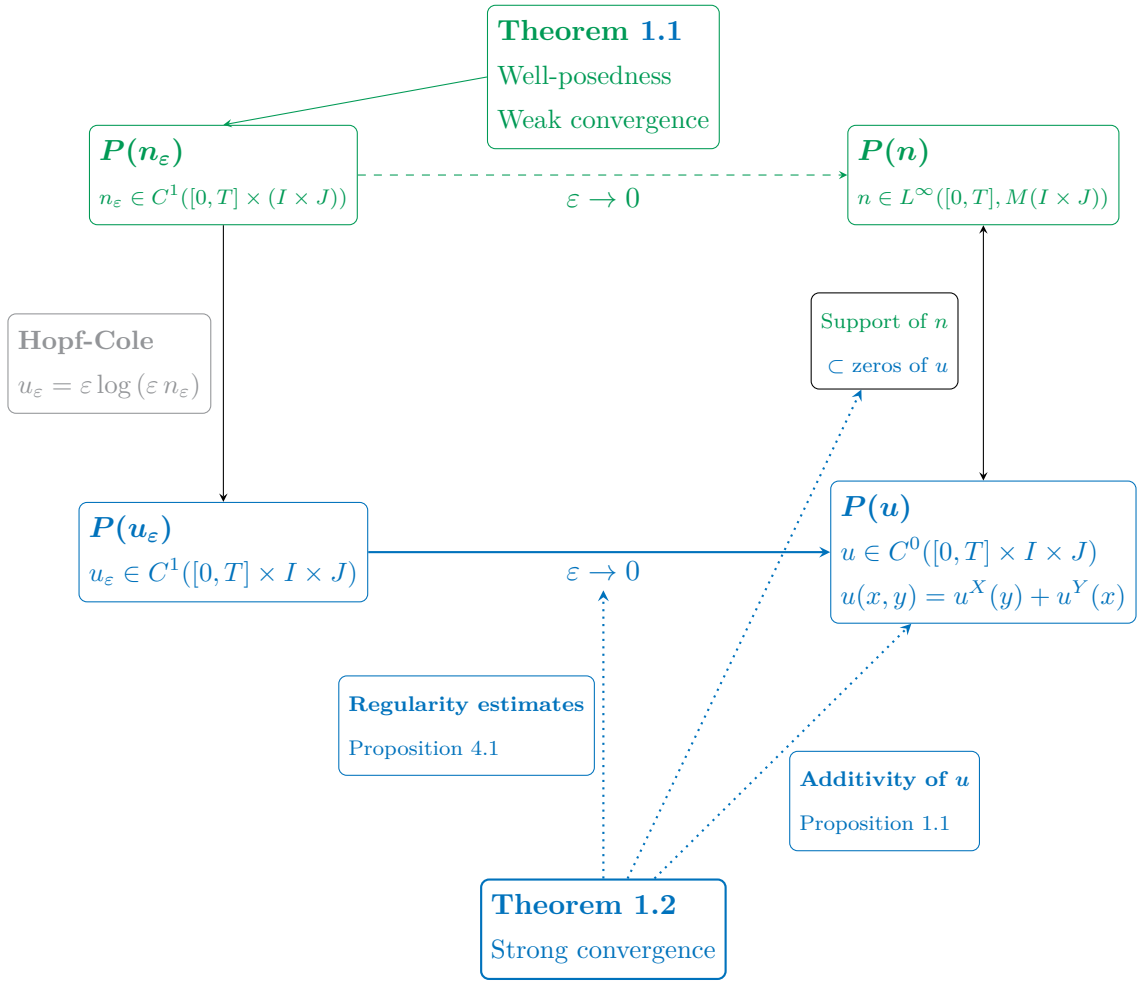
\begin{figure}
     \centering
    \begin{tikzpicture}
        \node[rounded corners=3pt, draw, color = ForestGreen] at (-5,0) (Pneps) {\begin{math}
    \begin{aligned}
    &\text{$\boldsymbol{\boldsymbol{P(n_\varepsilon)}}$}\\
    &\text{\scriptsize $n_\varepsilon \in C^1([0,T]\times(I\times J))$}
    \end{aligned}
    \end{math}};
        \node[rounded corners=3pt, draw, color = ForestGreen] at (5,0) (Pn) 
        {\begin{math}
    \begin{aligned}
    &\text{$\boldsymbol{\boldsymbol{P(n)}}$}\\
    &\text{\scriptsize $n \in L^\infty([0,T],M(I\times J))$}
    \end{aligned}
    \end{math}};
    \draw[-stealth, dashed, color = ForestGreen] (Pneps) -- (Pn);
    \node[rounded corners=3pt, draw, color = ForestGreen] at (0, 1.3) (CVn) {\begin{math}
    \begin{aligned}
    &\text{\textbf{\cref{thm:well_posedness}}}\\
    &\text{\small Well-posedness}\\
    &\text{\small Weak convergence}
    \end{aligned}
    \end{math}};
    \draw[-stealth,color = ForestGreen] (CVn.180) -- (Pneps.90);
    \node[color = ForestGreen] at (0, -0.3) {$\varepsilon \rightarrow 0$};
    \node[rounded corners=3pt, draw, color = RoyalBlue] at (-5,-5) (Pueps) {\begin{math}
    \begin{aligned}
    &\text{$\boldsymbol{P(u_\varepsilon)}$}\\
    &\text{\small $u_\varepsilon \in C^1([0,T]\times I\times J)$}    \end{aligned}
    \end{math}};
    \node[rounded corners=3pt, draw, color = RoyalBlue] at (5,-5) (Pu) {\begin{math}
    \begin{aligned}
    &\text{$\boldsymbol{P(u)}$}\\
    &\text{\small $u \in C^0([0,T]\times I\times J)$}\\
    &\text{\small $u(x,y) =u^X(y)+u^Y(x)$}\end{aligned}
    \end{math}};
    \draw[-stealth, thick, color = RoyalBlue] (Pueps) -- (Pu);
     \draw[-stealth] (Pneps) -- (Pueps);
     \node[rounded corners=3pt, draw, color =Gray] at (-6.5,-2.5) {\begin{math}
    \begin{aligned}
    &\text{\small\textbf{Hopf-Cole}}\\
    &\text{\small $u_\varepsilon = \varepsilon\log\left(\varepsilon\,n_\varepsilon\right)$}    \end{aligned}
    \end{math}};
    \node[rounded corners=3pt, draw, color =RoyalBlue, thick] at (0,-10) (thmu) {\begin{math}
    \begin{aligned}
    &\text{\textbf{\cref{thm:convergence_u}}}\\
    &\text{\small Strong convergence}   \end{aligned}
    \end{math}};
    \node[color = RoyalBlue] at (0, -5.25) (epsu) {$\varepsilon\rightarrow 0$};
    \draw[-stealth, thick, dotted, color = RoyalBlue] (thmu) -- (epsu);
    \node[rounded corners=3pt, draw, color =RoyalBlue] at (-1.8,-7.3) (cvu) {\begin{math}
    \begin{aligned}
    &\text{\scriptsize\textbf{Regularity estimates}}\\
    &\text{\scriptsize\ref{prop:uniform_estimates_u}}\end{aligned}
    \end{math}};
    \draw[-stealth, thick, dotted, color = RoyalBlue] (thmu) -- (Pu);
    \node[rounded corners=3pt, draw, color =RoyalBlue] at (3.7,-8.3) (cvu) {\begin{math}
    \begin{aligned}
    &\text{\scriptsize\textbf{Additivity of $u$}}\\
    &\text{\scriptsize\ref{prop:nu}}\end{aligned}
    \end{math}};
    \draw[stealth-stealth] (Pu) -- (Pn);
    \node[rounded corners=3pt, draw] at (3.7,-2.2) (zerossupport) {\begin{math}\begin{aligned}
    \text{\scriptsize\color{ForestGreen}Support of $n$}\\
    \text{\scriptsize\color{RoyalBlue}$\subset$ zeros of $u$}
    \end{aligned}\end{math}};
    \draw[-stealth, thick, dotted, color = RoyalBlue] (thmu) -- (zerossupport);
    \end{tikzpicture}     
\caption{Flowchart of the analytical results of \cref{sec:well_posedness}, \cref{sec:lipschitz} and \cref{sec:u}.}
     \label{fig:layout_convergence}
 \end{figure}

Moreover, although \ref{P_u} involves an equation on $u_\varepsilon$, one can notice that \cref{thm:convergence_u} states limit equations on $u^X$ and $u^Y$ \eqref{eq:limit_equation}. Instead of passing to the limit in the equation of $u$ in \ref{P_u} once the convergence is established (as it is done in most asexual studies in the regime of small variance), the separation of variables $u(x,y) = u^X(y) +u^Y(x)$ allows us to take {\color{black}an} alternative approach. In the proof \bl{of} \cref{thm:convergence_u}, we will show indeed that $u^X = \underset{\varepsilon\rightarrow 0}{\lim}\;\varepsilon\log\left(\rho^X_\varepsilon\right)$ and $u^Y = \underset{\varepsilon\rightarrow 0}{\lim}\;\varepsilon\log\left(\rho^Y_\varepsilon\right)$. The idea is then to focus on the equations satisfied by $\rho_\varepsilon^X$ and $\rho_\varepsilon^Y$ instead of the equation satisfied by $n_\varepsilon$:
\begin{equation}
    \begin{aligned}
    \begin{cases}
    \varepsilon\, \partial_t \rho^X_\varepsilon(t,y) = \left(r - \kappa\,\rho_\varepsilon(t)\right)\rho^X_\varepsilon(t,y) - \int_{I} m(x,y)\,n_\varepsilon(t,x,y) \,dx,\\
    \varepsilon\, \partial_t \rho^Y_\varepsilon(t,x) = \left(r - \kappa\,\rho_\varepsilon(t)\right)\rho^Y_\varepsilon(t,x) - \int_{J} m(x,y)\,n_\varepsilon(t,x,y) \,dy.
    \end{cases}
    \end{aligned}
    \label{eq:rhoX_rhoY}
\end{equation}
The advantage of considering \eqref{eq:rhoX_rhoY} over \eqref{P_u} is that the reproduction terms involved are linear, much simpler than the integral operator involved in the equation on $u_\varepsilon$. However, the difficulties are transferred on the \bl{selection terms} $\int_{I} m(x,y)\,n_\varepsilon(t,x,y) \,dx$ and $\int_{I} m(x,y)\,n_\varepsilon(t,x,y) \,dx$ that asymptotically lead to involve $\phi^X$ and $\phi^Y$ in \eqref{eq:limit_equation}. These terms are new compared to the typical asexual studies, which only present two unknown variables in their constrained limit equation: $u$ and $\rho$. Consequently, here, regularity in time, which would allow us to write the limit equation \eqref{eq:limit_equation} under a differential form, is harder to get for $\rho$ and $\phi^X$ (resp. $\phi^Y$).

Nevertheless, under an additional hypothesis on the selection term $m$ being additive, we show that the limit size of population $\rho$ is BV. This result aligns with the typical analogous regularity obtained on the asymptotic size of population in aforementioned asexual studies.
 
\begin{theorem}
\label{thm:BV}
Suppose that there exists $m^X:I\rightarrow \R$ and $m^Y:J\rightarrow \R$ such that: \begin{equation}
\label{eq:m_add}
\tag{$\text{H}_{m,\text{add}}$}
    m(x,y) = m^X(x)+m^Y({\color{black}y}).
\end{equation}
Let $n_\varepsilon$ be the solution to $\cref{P_n}$. Then, $\rho_\varepsilon$ is locally uniformly bounded in $W^{1,1}(\R_+)$.
Consequently, after extraction of a subsequence, $\rho_\varepsilon$ converges to a BV-function $\rho$ as $\varepsilon$ vanishes. The limit $\rho$ is non-decreasing as soon as there exists a constant $C>0$ such that:\begin{equation}
\label{eq:I_0}
(r-\kappa\,\rho^0_\varepsilon)\,\rho^0_\varepsilon -\iint_{\R^2}m(x,y)\,n^0_\varepsilon(x,y)\,dx\,{\color{black}dy} \geq -C\,e^{\frac{o(1)}{\varepsilon}}.
\end{equation}
\end{theorem}

The paper is organized as follows. In \cref{sec:formal}, we present qualitative results and numerical analysis that stem from the analysis of the subsequent sections, and demonstrate the interest of the model by exploring some biologically relevant situations. Next, in \cref{sec:well_posedness}, we prove the well-posedness of \ref{P_n}.  \cref{sec:lipschitz} is dedicated to show \ref{prop:nu} and derive uniform $L^\infty$ and Lipschitz bounds for $u_\varepsilon$, which prepares the proof of the main result in \cref{sec:u}. The interplay between the different results until that point is displayed in \cref{fig:layout_convergence}. Finally, in \cref{sec:BV}, we show that $\rho$ is a BV-function, under the additional hypothesis \eqref{eq:m_add}.

\section{Qualitative results and numerical analysis}
\label{sec:formal}
In this section, we explore the insights on the dynamics of the allelic distribution in a population following the main result of the paper (\cref{thm:convergence_u}), assuming that $(t,y)\mapsto\phi^X(t,\cdot,y)$, $(t,x)\mapsto\phi^Y(t,x,\cdot)$ and $t\mapsto n(t,\cdot,\cdot)$ (and by extension $t\mapsto \rho^X(t,\cdot)$, $t\mapsto \rho^Y(t,\cdot)$ and $\rho$) are continuous so that we can formally write:
\begin{equation}
\begin{aligned}
\begin{cases}
    \forall (t,y) \in [0,T]\times J,\quad \partial_t u^X(t,y) = r-\kappa\,\rho (t) - \left\langle \phi^X(t,\cdot,y),m(\cdot,y)\right\rangle,\\
    \forall (t,x) \in [0,T]\times I,\quad \partial_t u^Y(t,x) = r-\kappa\,\rho (t) - \left\langle \phi^Y(t,x,\cdot),m({\color{black}x},\cdot)\right\rangle.
\end{cases}
\end{aligned}
    \label{eq:formal_limit_equation_haploid}
\end{equation}
We first show that under a hypothesis of strict monotony \ref{hyp:monotony_m} on the selection, the population is strained to be monomorphic, i.e. all individuals share the same alleles $\left(\bar{x}(t),\bar{y}(t)\right)$ at all times. Then, we derive canonical equations describing the dynamics of $\left(\bar{x}(t),\bar{y}(t)\right)$ under monomorphism.
\subsection{Monotonic selection yields monomorphism}
We first show that a condition of monotony on $m$ (in both variables) yields the limit allelic distribution to be monomorphic at all times:

\begin{prop} For $T>0$, assume that \eqref{eq:formal_limit_equation_haploid} holds and that $m$ satisfies:
\begin{equation}
\label{hyp:monotony_m}
    \tag{$\text{H}_\text{increasing}$}
    \forall (x,y) \in I\times J, m(x,\cdot) \text{ and } m(\cdot,y) \text{ are increasing (resp. decreasing).}
\end{equation}
Then\bl{, the following holds}: $\forall t \in [0,T], \; \exists !\, (\bar{x}(t),\bar{y}(t)) \in I\times J,$
\begin{align*}
    \text{Supp}\left(\rho^X(t,\cdot)\right)=\left[u^X(t,\cdot)\right]^{-1}\left(\{0\}\right)=\{\bar{y}(t)\},\quad\text{Supp}\left(\rho^Y(t,\cdot)\right)=\left[u^Y(t,\cdot)\right]^{-1}\left(\{0\}\right)=\{\bar{x}(t)\}.
\end{align*}
\label{prop:yield_monomorphism}
\end{prop}
{\color{black}\paragraph{Diploid case: homozygosity.}In the diploid case, the symmetries indicated in \cref{rem:diploid} yield $u^X = u^Y$ and therefore $\bar{x}(t)=\bar{y}(t)$ for all $t\in[0,T]$. All individuals are therefore homozygote in a monomorphic population.}
\begin{proof}[Proof of \ref{prop:yield_monomorphism}]
For $t\in [0,T]$, since $n(t,\cdot,\cdot)$ is supported at the zeros of $u(t,\cdot,\cdot)$ (see \cref{thm:convergence_u}), $\rho^X(t,\cdot)$ is supported on the set of the zeros of $u^X(t,\cdot)$, that we denote {\color{black}by} $F_X(t)$, and $\rho^Y(t,\cdot)$ is supported on the set of the zeros of $u^Y(t,\cdot)$, that we denote {\color{black}by} $F_Y(t)$.  It is therefore sufficient to prove that $F_X(t)$ and $F_Y(t)$ are both singletons for all $t\in[0,T]$.

\bl{Since $u^X(t,\cdot)$ and $u^Y(t,\cdot)$ are continuous}, $F_X(t)$ and $F_Y(t)$ are closed subsets of $I$ and $J$, and are therefore compact sets. In particular, the extreme points of $F_X(t)$ (resp. $F_Y(t)$) denoted by $y_\text{inf}(t)$ and $y_\text{sup}(t)$ (respectively, $x_\text{inf}(t)$ and $x_\text{sup}(t)$) lie in $F_X(t)$ (respectively, $F_Y(t)$).  As $(t,y_\text{inf}(t))$ and $(t,y_\text{sup}(t))$ maximise $u^X$ and $(t,x_\text{inf}(t))$ and $(t,x_\text{sup}(t))$ maximise $u^Y$ (since $u^X$ are $u^Y$ are non-positive, from \cref{thm:convergence_u}), we obtain that \[0= \partial_t u^X(t,y_\text{inf}(t)) =\partial_t u^X(t,y_\text{sup}(t))=\partial_t u^Y(t,x_\text{inf}(t)) =\partial_t u^Y(t,x_\text{sup}(t)).\] 

The equations \cref{eq:formal_limit_equation_haploid} next implies that
\begin{align*}
    \forall t\in[0,t], \left\langle \phi^X(t,\cdot,y_\text{inf}(t)),m(\cdot,y_\text{inf}(t))\right\rangle= \left\langle \phi^X(t,\cdot,y_\text{sup}(t)),m(\cdot,y_\text{sup}(t))\right\rangle\\ = \left\langle \phi^Y(t,x_\text{inf}(t),\cdot),m(x_\text{inf}(t),\cdot)\right\rangle = \left\langle \phi^Y(t,x_\text{sup}(t),\cdot),m(x_\text{sup}(t),\cdot)\right\rangle.
\end{align*}
Recall that, for $(t,x,y)\in[0,T]\times I\times J$, $\phi^X(t,\cdot,y)$ and $\phi^X(t,x,\cdot)$ are probability distributions supported respectively on a subset of $F_X(t)$ and $F_Y(t)$ (from \cref{thm:convergence_u}). Then, we deduce from \ref{hyp:monotony_m} that:
\begin{align*}
    m\left(x_\text{inf}(t),y_\text{sup}(t)\right) &\leq \left\langle \phi^X(t,\cdot,y_\text{sup}(t)),m(\cdot,y_\text{sup}(t))\right\rangle\\
    &= \left\langle \phi^Y(t,x_\text{inf}(t),\cdot),m(x_\text{inf}(t),\cdot)\right\rangle\\
    & \leq m\left(x_\text{inf}(t),y_\text{sup}(t)\right).
\end{align*}
Similarly, we obtain
\begin{align*}
    m\left(x_\text{sup}(t),y_\text{inf}(t)\right) &{\color{black}\leq} \left\langle \phi^Y(t,x_\text{sup}(t),\cdot),m(x_\text{sup}(t),\cdot)\right\rangle\\
    &= \left\langle \phi^X(t,\cdot,y_\text{inf}(t)),m(\cdot,y_\text{inf}(t))\right\rangle\\
    & \leq m\left(x_\text{sup}(t),y_\text{inf}(t)\right).
\end{align*}
All the inequalities above must be equalities, which implies 
\begin{equation}
\begin{aligned}
    &\phi^X(t,x,y_{\sup{}}(t)) = \delta_{x = x_{\inf}(t)},\quad \phi^X(t,x,y_{\inf{}}(t)) = \delta_{x = x_{\sup}(t)},\\
    &\phi^Y(t,x_{\sup{}}(t),y) = \delta_{y = y_{\inf}(t)},\quad \phi^Y(t,x_{\inf{}}(t),y) = \delta_{y = y_{\sup}(t)}.
\end{aligned}
\label{eq:phi_dirac}
\end{equation}
Since the support of $\phi^X(t,\cdot,y)$ (resp. $\phi^Y(t,x,\cdot)$) does not depend on $y$ (resp. $x$) (see $(iv)$ of \cref{thm:convergence_u}), we obtain from \eqref{eq:phi_dirac} that $x_{\inf}(t) = x_{\sup}(t)$ and $y_{\inf}(t)=y_{\sup}(t)$. The latter yields the result.
\end{proof}

\paragraph{Numerical simulations: robustness of monomorphism with regard to \ref{hyp:monotony_m} of \ref{prop:monomorphism} with dimorphic initial densities.}

We show in \cref{fig:monomorphism_robustness} the result of numerical simulations solving a discretized version of \eqref{P_n} \bl{with initial dimorphic} states to test the robustness of monomorphic trajectories with regard to \ref{hyp:monotony_m}. We consider three different {\color{black}selection functions} $m(x,y) = x^2+y^2$, $m(x,y) = (x+y)^2$, $m(x,y) = (1-xy)^2$. \cref{fig:monomorphism_robustness} seems to indicate that monomorphic trajectories occur under a wider scope \bl{than the one required by} \ref{prop:yield_monomorphism}. \cref{fig:monomorphism_robustness} also gives some insights on the diversity of trajectories that can arise under different \bl{selection functions} (see the next subsection for a more complete view).
\begin{figure}
\centering
        \begin{subfigure}{.45\textwidth}
            \centering
            \includegraphics[height=0.25\textheight]{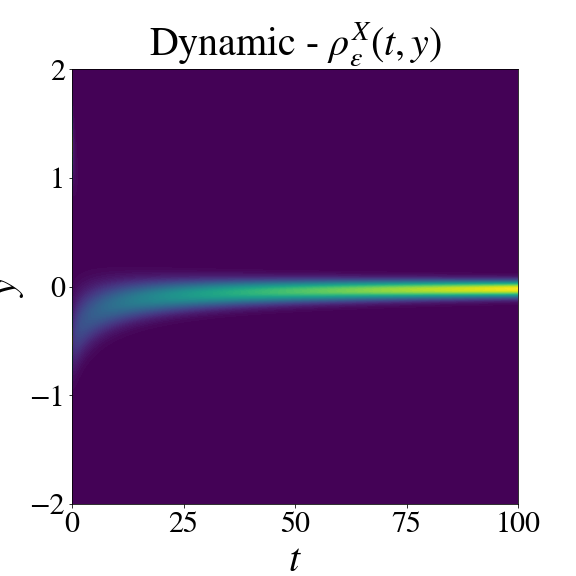}
            \subcaption{$m(x,y) = x^2+y^2$}
            \label{fig:sum_of_squares_0_1}
        \end{subfigure}
        \begin{subfigure}{.45\textwidth}
           \centering \includegraphics[height=0.25\textheight]{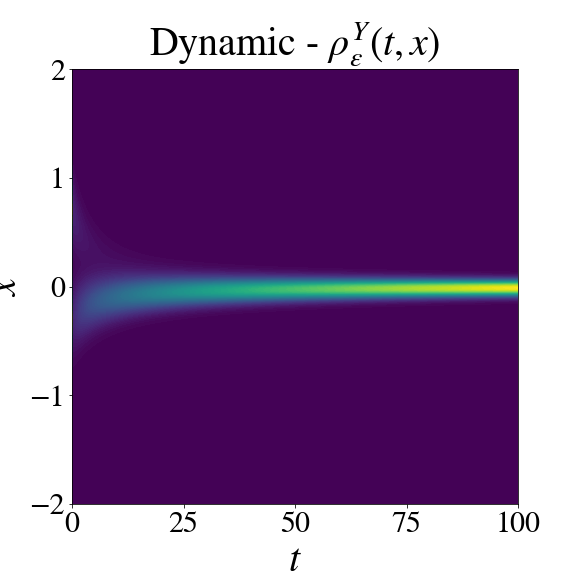}
            \subcaption{$m(x,y) = x^2+y^2$.}
            \label{fig:sum_of_squares_-2_2}
        \end{subfigure}\\
        \begin{subfigure}{.45\textwidth}
            \centering
            \includegraphics[height=0.25\textheight]{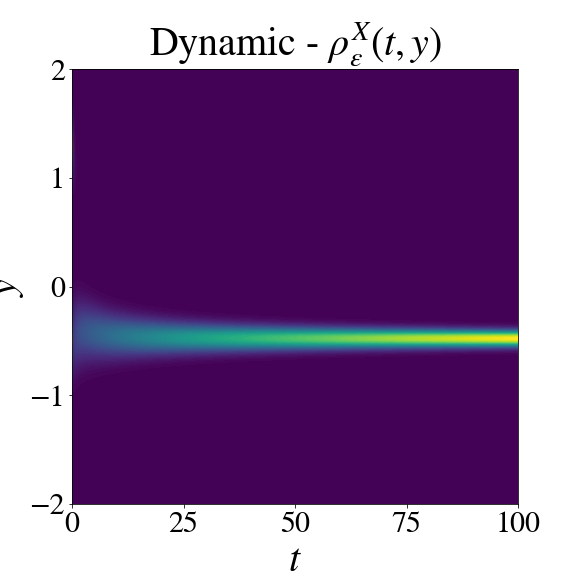}
            \subcaption{$m(x,y) = (x+y)^2.$}
            \label{fig:squared_sum_0_1}
            \end{subfigure}
            \begin{subfigure}{.45\textwidth}
            \centering
            \includegraphics[height=0.25\textheight]{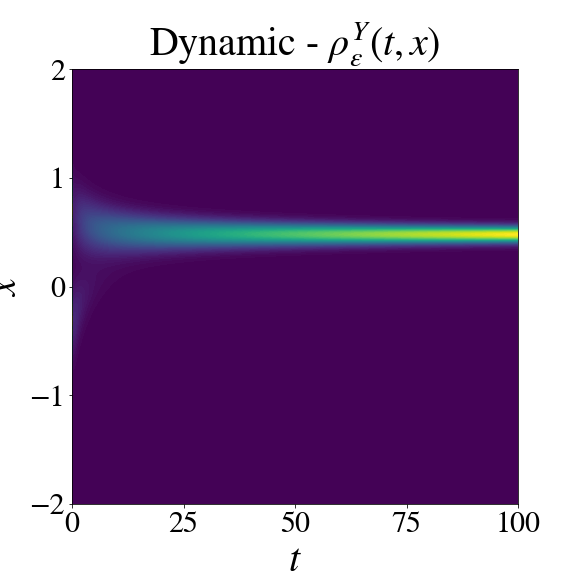}
            \subcaption{$m(x,y) = (x+y)^2.$}
            \label{fig:squared_sum_-2_2}
        \end{subfigure}\\
        \begin{subfigure}{.45\textwidth}
            \centering
            \includegraphics[height=0.25\textheight]{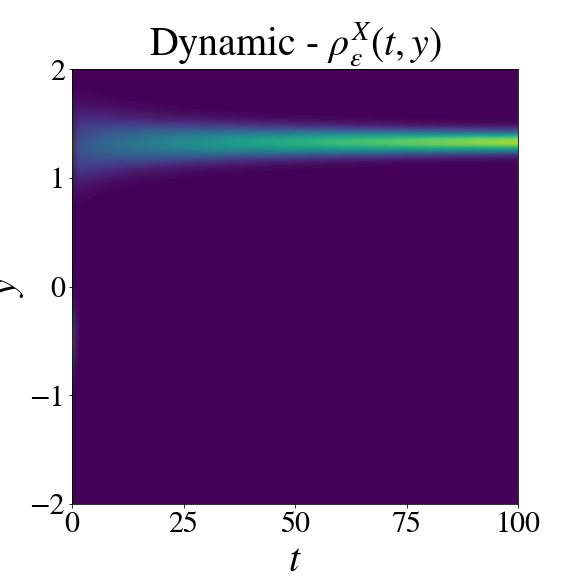}
            \subcaption{$m(x,y) = (1-x\,y)^2.$}
            \label{fig:squared_sum_0_1}
            \end{subfigure}
            \begin{subfigure}{.45\textwidth}
            \centering
            \includegraphics[height=0.25\textheight]{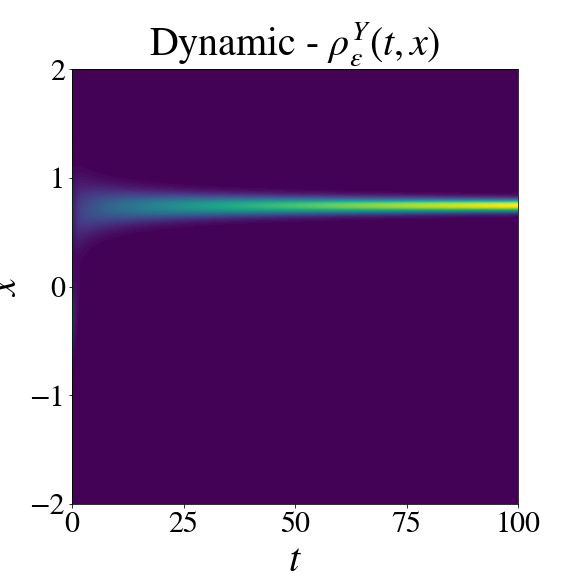}
            \subcaption{$m(x,y) = (1-x\,y)^2$.}
            \label{fig:squared_sum_-2_2}
        \end{subfigure}
    \caption{\textbf{Robustness of monomorphism with regard to {\color{black}assumption} \ref{hyp:monotony_m} of \ref{prop:monomorphism} with dimorphic initial densities.} For each selection function (by row), we display the numerically solved dynamics of $\rho^X(t,y)$ (left panel) and $\rho^Y(t,x)$ (right panel) ($(x,y) \in [-2,2]$). {\color{Black} The colors correspond to isolines of $\rho^X$ and $\rho^Y$.} {\color{black}The initial state is sum of two Gaussians centered in $(x_1,y_1) = (-0.3, 1.3)$ and $(x_2,y_2) = (0.7, -0.5)$ and of variance $\varepsilon = 0.05$}. Lighter colors indicate stronger densities. The figures seem to indicate that the trajectories become monomorphic almost instantaneously and that this phenomenon actually occur under {\color{black}weaker} conditions than \ref{hyp:monotony_m} of \ref{prop:monomorphism}. One can also notice that the stationary dominant alleles that arise vary greatly from one selection function to another.}
    \label{fig:monomorphism_robustness}
\end{figure}

\subsection{Canonical equations under monomorphism}

In all this section, let us fix $T>0$ and let us assume that for all time $t\in[0,T]$, there exists {\color{black}unique points $\bar{x}(t)$ and $\bar{y}(t)$} such that:
\begin{equation}
    \forall t \in [0,T]\quad u(t,\cdot,\cdot)^{-1}\left(\{0\}\right) = \{(\bar{x}(t),\bar{y}(t))\}.
    \label{eq:monomorphism_hyp}
\end{equation}
In that case, for all $(t,x,y) \in [0,T]\times I\times J$, we deduced from \cref{thm:convergence_u} that:
\begin{equation*}
    \phi^X(t,\cdot,y) = \delta_{\bar{x}(t)}, \quad \phi^Y(t,x,\cdot) = \delta_{\bar{y}(t)}.
\end{equation*}
Hence, \eqref{eq:formal_limit_equation_haploid} reads:
\begin{equation}
\begin{aligned}
\begin{cases}
    \forall (t,y) \in [0,T]\times J,\quad \partial_t u^X(t,y) = r-\kappa\,\rho (t) - \bl{m(\bar{x}(t),y)},\\
    \forall (t,x) \in [0,T]\times I,\quad \partial_t u^Y(t,x) = r-\kappa\,\rho (t) - m(x,\bar{y}(t)).
\end{cases}
\end{aligned}
    \label{eq:formal_limit_equation_haploid_monom}
\end{equation}
\begin{prop}
\label{prop:monomorphism}
For $T>0$, {\color{black}assume that $m, u_0 \in C^2(I\times J)$ and that \eqref{eq:formal_limit_equation_haploid_monom} holds.}
Then: $u\in C^0([0,T],C^2(I\times J))$ and the dynamics of the dominant alleles $(\bar{x}(t),\bar{y}(t))$ read:
\begin{equation}
\begin{aligned}
\begin{cases}
{\partial_{xx} u^Y(\bar{x}(t))}\,\frac{d\bar{x}}{dt} =\partial_x m(\bar{x}(t),\bar{y}(t)),
\\
{\partial_{yy} u^X(\bar{y}(t))}\,\frac{d\bar{y}}{dt} =\partial_y m(\bar{x}(t),\bar{y}(t)).
\end{cases}
\end{aligned}
    \label{eq:canonical_equation}
\end{equation}
\end{prop}

{\color{black}\paragraph{Diploid case: canonical equations.} In the diploid case, the symmetries indicated in \cref{rem:diploid} yield $u^X = u^Y$ and $\bar{x}=\bar{y}$, so the canonical equations \eqref{eq:canonical_equation} reduce to
\begin{equation}
    \label{eq:canonical_diploid}
    {\partial_{xx} u^X(\bar{x}(t))}\,\frac{d\bar{x}}{dt} =\partial_x m(\bar{x}(t),\bar{x}(t)).
\end{equation}
}
\begin{proof}
Let us show {\color{black}how to obtain} the first equation of \eqref{eq:canonical_equation} on $\bar{x}(t)${\color{black}. The equation on $\bar{y}(t)$ can be obtained similarly}.

As $0 = \partial_x u^Y(t,\bar{x}(t))=u^Y(t,\bar{x}(t)) = \max u^Y(t,\cdot) $ for all $t \in[0,T]$, we get:
\[0 = \frac{d\,\partial_x u^Y(t,\bar{x}(t))}{dt} = \partial_t\partial_x u^Y(t,\bar{x}(t)) + \partial_{xx}u^Y(t,\bar{x}(t))\,\frac{d\bar{x}}{dt}.\]
Differentiating \eqref{eq:formal_limit_equation_haploid_monom} with regard to $x$ reads:
\begin{equation*}
    \forall (t,x) \in [0,T]\times I,\quad \partial_x\partial_t u^Y(t,x) =  - \partial_x m(x,\bar{y}(t)).
\end{equation*}
By substitution, we obtain:
\[    \frac{d\bar{x}}{dt} \,\partial_{xx}u^Y(t,\bar{x}(t)) = \partial_x m(\bar{x}(t),\bar{y}(t)).\]
\end{proof}
\begin{rem}
As $\bar{x}(t)$ maximizes $u^Y$, we have $\partial_{xx}u^Y(t,\bar{x}(t)) \leq 0$ for all $t\in[0,T]$. If $\partial_{xx}u^Y(t,\bar{x}(t))<0$ for all $t\in[0,T]$, then we obtain:
\[ \forall t \in [0,T],\quad   \frac{d\bar{x}}{dt}  = \frac{\partial_x m(\bar{x}(t),\bar{y}(t))}{\partial_{xx}u^Y(t,\bar{x}(t))}.\]
As $\partial_t\partial_{xx} u^Y(t,x) = -\partial_{xx} m(x,\bar{y}(t))$, we obtain 
\begin{equation*}
    \partial_{xx} u^Y(t,x) = \partial_{xx} u^Y(0,x) - \int_0^t \partial_{xx} m(x,\bar{y}(s))\,ds.
\end{equation*}
Consequently, the strict inequality is ensured if $u_0$ is strictly concave and $m$ is convex.
\end{rem}

\paragraph{Three examples.}

In this paragraph, we illustrate the insights provided by \ref{prop:monomorphism} through the study of the system for three given selection functions $m$. In all examples, we consider that $I = J = [-2,2]$ and the initial state $u_0$ is given by $(\bar{x}_0, \bar{y}_0)\in [-2,2]^2$ and:
\begin{equation*}
    u_{0}(x,y) = -((x-\bar{x}_0)^2+(y-\bar{y}_0)^2).
\end{equation*}

\emph{1) $m(x,y) = x^2+y^2,\quad \partial_x m(x,y) = 2x, \quad \partial_y m(x,y) = 2y, \quad\partial_{xx} m(x,y) = \partial_{yy} m(x,y) = 2$.}

This selection function separates additively {\color{black}the} variables. The canonical equation \cref{eq:canonical_equation} then reads:
\begin{equation*}
\begin{aligned}
\begin{cases}
    \frac{d \bar{x}(t)}{dt} = -\frac{\bar{x}(t)}{t+1},\\
    \frac{d \bar{y}(t)}{dt} = -\frac{\bar{y}(t)}{t+1}.
\end{cases}
\end{aligned}
\end{equation*}
We obtain that, for $t\geq 0$
\begin{equation*}
        \bar{x}(t) = \frac{\bar{x}_0}{t+1},\quad \bar{y}(t)= \frac{\bar{y}_0}{t+1}.
\end{equation*}

{\color{black}Consequently, the system remains monomorphic and the dominant alleles evolve and converge \bl{toward} $(0,0)$.}

\emph{2) $m(x,y) = (x+y)^2,\quad \partial_x m(x,y) = \partial_y m(x,y) = 2(x+y), \quad \partial_{xx} m(x,y) = \partial_{yy} m(x,y) = 2$.}

The canonical equation \cref{eq:canonical_equation} then reads:
\begin{equation*}
    \frac{d \bar{x}(t)}{dt} = \frac{d \bar{y}(t)}{dt}= -\frac{\bar{x}(t)+\bar{y}(t)}{t+1}.
\end{equation*}
We deduce that, for $t\geq 0$
\begin{equation*}
    \bar{x}(t) + \bar{y}(t) = \frac{\bar{x}_0+\bar{y}_0}{(t+1)^2},\quad \bar{x}(t) - \bar{y}(t) = \bar{x}_0-\bar{y}_0,
\end{equation*}
which leads to:
\begin{equation*}
    \bar{x}(t) = \frac{\bar{x}_0-\bar{y}_0}{2}+\frac{\bar{x}_0+\bar{y}_0}{2(t+1)^2},\quad \bar{y}(t) = \frac{\bar{y}_0-\bar{x}_0}{2}+\frac{\bar{x}_0+\bar{y}_0}{2(t+1)^2}.
\end{equation*}
On the contrary to the previous example, {\color{black}the dominant alleles of the monomorphic system evolve to converge \bl{toward} a state that is} dependent on the initial state of the system. Geometrically, it is the orthogonal projection of the initial point $(\bar{x}_0, \bar{y_0})$ on the diagonal defined by $x+y=0$. 

\emph{3) $m(x,y) = (1-xy)^2,\quad \partial_x m(x,y) = -2\,y\,(1-xy),\quad \partial_y m(x,y) = -2\,x\,(1-xy)$,}

\emph{$\partial_{xx} m(x,y) = 2\,y^2,\quad \partial_{yy} m(x,y) = 2\,x^2$.}

In {\color{black}this} case, the canonical equation \eqref{eq:canonical_equation} reads:
\begin{equation}
    \frac{d \bar{x}(t)}{dt} = \frac{\bar{y}(t)\left(1-\bar{x}(t)\bar{y}(t)\right)}{1+\int_0^t\bar{y}(s)^2\,ds}\qquad \frac{d \bar{y}(t)}{dt} = \frac{\bar{x}(t)\left(1-\bar{x}(t)\bar{y}(t)\right)}{1+\int_0^t\bar{x}(s)^2\,ds}.
    \label{eq:(1-xy)^2}
\end{equation}
Without lack of generality, we can assume that $\bar{x}_0 \leq \bar{y}_0$.
\begin{prop}
Let $0< \bar{x}_0 \leq \bar{y}_0\leq 2$. Then the {\color{black}dominant alleles of the monomorphic system converge} \bl{toward} the stationary state $(x_F,y_F) \in \left(\R^*_+\right)^2$ that solves:
\begin{equation}
    \begin{aligned}
        \begin{cases}
        x_F\,y_F = 1,\\
        y_F^2-x_F^2 = \bar{y}_0^2-\bar{x}_0^2.
        \end{cases}
    \end{aligned}
    \label{eq:system_x_F_Y_f}
\end{equation}
\label{prop:1-xy}
\end{prop}
\begin{proof}
{\color{black}First, we treat the case where $x_0\,y_0=1$. Then, the function $t\mapsto (x_0,y_0)$ defines a solution of \eqref{eq:(1-xy)^2}. By {\color{black}uniqueness}, it is the only solution and $(x_0,y_0)$ satisfies \eqref{eq:system_x_F_Y_f}.}

Let us now suppose that $x_0\,y_0<1$ (the case where $x_0\,y_0>1$ {\color{black}can be treated following similar arguments}). {\color{black}We define}, for $A>0$ yet to be specified:
\[t_A = \min\left\{\inf\{t\geq 0, (\bar{x}(t),\bar{y}(t)) \notin ]0,A[^2\},\inf\{t\geq 0, \bar{x}(t)\bar{y}(t) \notin ]0,1[\}\right\}.\]
For $t\leq t_A$, we have the following inequalities
\begin{equation}\frac{\bar{y}\left(1-\bar{x}\bar{y}\right)}{1+A^2\,t}\leq \frac{d \bar{x}}{dt}\leq \bar{y}\left(1-\bar{x}\bar{y}\right),\quad \frac{\bar{x}\left(1-\bar{x}\bar{y}\right)}{1+A^2\,t}\leq \frac{d \bar{y}}{dt}\leq \bar{x}\left(1-\bar{x}\bar{y}\right).
\label{eq:inequality_x_bar}
\end{equation}
Let us define $(x^-,y^-)$ and $(x^+,y^+)$ solutions of the {\color{black}following equations}:
\begin{equation}
    \begin{aligned}
    \begin{cases}
    \frac{d x^-(t)}{dt} = \frac{y^-(t)\left(1-x^-(t)y^-(t)\right)}{1+A^2\,t},\quad &\frac{d x^+(t)}{dt} = y^+(t)\left(1-x^+(t)y^+(t)\right),\\
    \frac{d y^-(t)}{dt} = \frac{x^-(t)\left(1-x^-(t)y^-(t)\right)}{1+A^2\,t},\quad &\frac{d y^+(t)}{dt} = x^+(t)\left(1-x^+(t)y^+(t)\right),\\
    (x^-(0),y^-(0){\color{black})} = (\bar{x}_0,\bar{y}_0),\quad &(x^+(0),y^+(0){\color{black})} = (\bar{x}_0,\bar{y}_0).
    \end{cases}
    \end{aligned}
    \label{eq:sub_sup_solutions_x_y}
\end{equation}
By comparison, we deduce that $(x^-,y^-)$ and $(x^+,y^+)$ are respectively subsolution and supersolution of $(\bar{x},\bar{y})$:
\begin{equation}
    \forall t \leq t_A, \quad x^-(t) \leq \bar{x}(t)\leq x^+(t),\quad y^-(t) \leq \bar{y}(t)\leq y^+(t).
    \label{eq:sub_sup_comparison}
\end{equation}
{\color{black}We} define $t_A^+$ by
\[t_A^+ = \min\left\{\inf\{t\geq 0, (x^+(t),y^+(t)) \notin ]0,A[^2\},\inf\{t\geq 0, x^+(t)y^+(t) \notin ]0,1[\}\right\}.\]
{\color{black}We will} show that $x^+y^+$ {\color{black}converges} increasingly toward 1. First one can compute that
\begin{equation*}
    \frac{d\left(x^+y^+\right)}{dt} = \left(1-x^+y^+\right)\left({x^+}^2+{y^+}^2\right).
\end{equation*}
Next, one can notice from \eqref{eq:sub_sup_solutions_x_y} that $x^+$ and $y^+$ both increase on $[0,t_A^+]$. We thus obtain that for $t\in[0,t_A^+]$
\begin{equation*}
    \left(1-x^+y^+\right)\left({x_0^2+y_0^2}\right) \leq \frac{d(x^+y^+)}{dt} \leq 2A^2\left(1-x^+y^+\right).
\end{equation*}
Hence, by comparison, $x^+y^+$ {\color{black}converges} increasingly toward 1.

We next notice {\color{black}thanks to \eqref{eq:sub_sup_solutions_x_y} that}: 
\begin{equation*}
    \frac{1}{2}\frac{d {x^+}^2(t)}{dt} = x^+(t)y^+(t)\left(1-x^+(t)y^+(t)\right) = \frac{1}{2}\frac{d {y^+}^2(t)}{dt},
\end{equation*}
and hence:
\begin{equation*}
    \forall t \leq t_A^+,\quad {{x^{+}}^2}(t) - {y^{+}}^2(t) = \bar{x}_0^2 - \bar{y}_0^2.
\end{equation*}
Therefore, since $(\bar{x}_0,\bar{y}_0) \in [0,2]^2$, and {\color{black}$0< x^+(t)\,y^+(t)< 1$} for $t\leq t_A^+$, the latter implies that, {\color{black}if we choose $A$ large enough}, $x^+$ and $y^+$ {\color{black}remain} uniformly bounded {\color{black}above away from} $A$. {\color{black}Therefore, we can} consider $t_A^+$ arbitrarily large. We deduce that $x^+$ and $y^+$ converge increasingly to $x_F >0$ and $y_F>0$, satisfying \eqref{eq:system_x_F_Y_f}.

We next show that $(x^-,y^-)$ converges toward the same couple $(x_F,y_F)$. Notice that for $t\leq t_A^+$, we have
\[\left(x^-(t),y^-(t)\right) \in ]0,A[^2, \quad x^-(t),y^-(t)\leq 1.\]
Similarly as previously, we show by comparison that $x^-\,y^-$ converges increasingly toward 1, since
\begin{equation*}
    \left(1-x^-y^-\right)\left(\frac{x_0^2+y_0^2}{1+A^2t_A^+}\right) \leq \frac{d(x^-y^-)}{dt} \leq 2A^2\left(1-x^-y^-\right).
\end{equation*}
Next, we notice that we still have: \[\frac{d {x^-}^2(t)}{dt} = \frac{d {y^-}^2(t)}{dt} \implies \forall t \leq t_A^+,\quad {{x^{-}}^2}(t) - {y^{-}}^2(t) = \bar{x}_0^2 - \bar{y}_0^2.\]
We deduce that $x^-$ and $y^-$ converge also increasingly to a solution of \eqref{eq:system_x_F_Y_f}. As \eqref{eq:system_x_F_Y_f} has a unique solution in $\left(\R^*_+\right)^2$, it must be $(x_F,y_F)$.

Finally, we obtain the announced result using \eqref{eq:sub_sup_comparison}.
\end{proof}

{\color{black}\paragraph{Diploid case with the three selection functions.}
Due to the symmetries indicated in \cref{rem:diploid}, the dynamics of the dominant allele \eqref{eq:canonical_diploid} are simpler, because they are limited to occur on the diagonal $x=y$.

1) $m(x,y) = x^2+y^2$.
We obtain that, for $t\geq 0$
\begin{equation*}
        \bar{x}(t)= \bar{y}(t) = \frac{\bar{x}_0}{t+1}.
\end{equation*}
Consequently, the system remains monomorphic and the dominant alleles evolve and converge \bl{toward} $(0,0)$.

2) $m(x,y) = (x+y)^2$.
We obtain that, for $t\geq 0$
\begin{equation*}
        \bar{x}(t) = \bar{y}(t) = \frac{\bar{x}_0}{(t+1)^2}.
\end{equation*}
Consequently, the system remains monomorphic and the dominant alleles evolve and converge \bl{toward} $(0,0)$.

3) $m(x,y) = (1-xy)^2$.
In that case, we deduce from \ref{prop:1-xy} that the dominant alleles of the monomorphic system converge toward $(1,1)$.
}

\paragraph{Numerical analysis.}
{\color{black}Note that \ref{prop:monomorphism} relies on the fact that  equation \eqref{eq:formal_limit_equation_haploid_monom} holds. Due to lack of regularity estimates, in this paper we have proved this property only in a weaker integral form \eqref{eq:limit_equation}. However, we conjecture that this property would hold in a rather general framework.
In \cref{fig:domninant_alleles_trajectories} using numerical simulations, we investigate whether the qualitative results obtained above are consistent in the case of the three examples considered in \cref{fig:monomorphism_robustness}.}
For each selection function above, we display the trajectories of the dominant allelic effects $\bar{x}$ and $\bar{y}$, for 20 numerical \bl{resolutions of} \cref{P_n} with $\varepsilon = 0.01$ (plain lines), with initial conditions uniformly randomized over the square $[-2,2]^2$ (each color corresponds to an initial condition). We confront them to the canonical equations given in \ref{prop:monomorphism}, for the same set of 20 initial conditions (dashed lines). The corresponding trajectories as well as the final states (full circle for the model and cross for the canonical equations) {\color{black}are quite in agreement}.

\begin{figure}
\centering
        \begin{subfigure}{.45\textwidth}
            \centering
            \includegraphics[width=\linewidth]{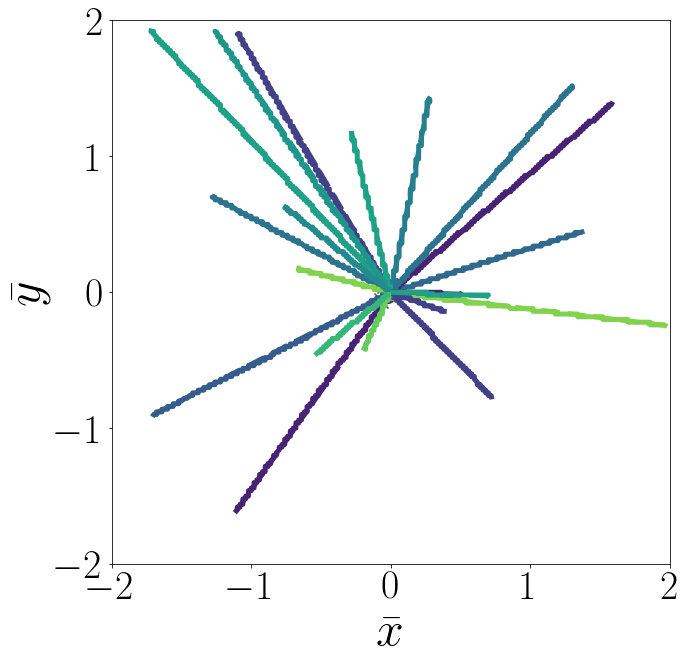}
            \subcaption{$m(x,y) = x^2+y^2.$}
            \label{fig:sum_of_squares}
        \end{subfigure}
        \begin{subfigure}{.45\textwidth}
            \includegraphics[width=\linewidth]{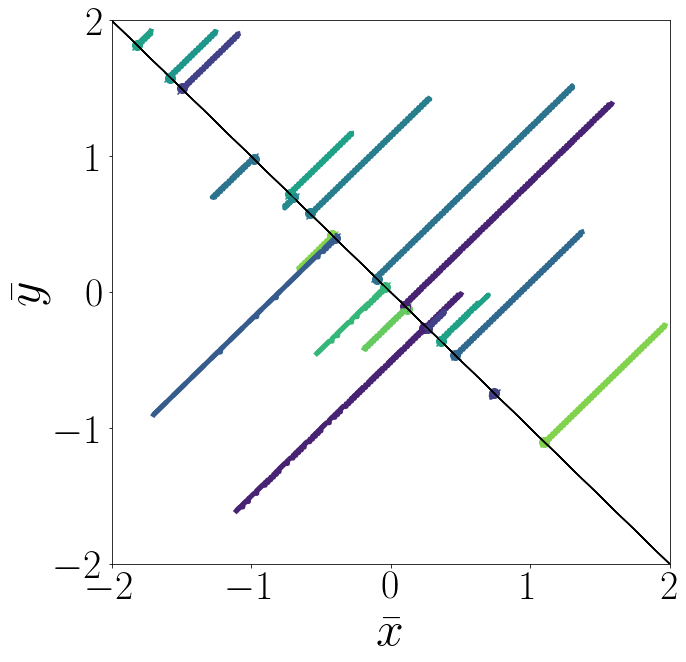}
            \subcaption{$m(x,y) = (x+y)^2.$ (the black line represents the line $x+y=0$).}
            \label{fig:squared_sum}
        \end{subfigure}\\
        \begin{subfigure}{.45\textwidth}
            \centering
            \includegraphics[width=\linewidth]{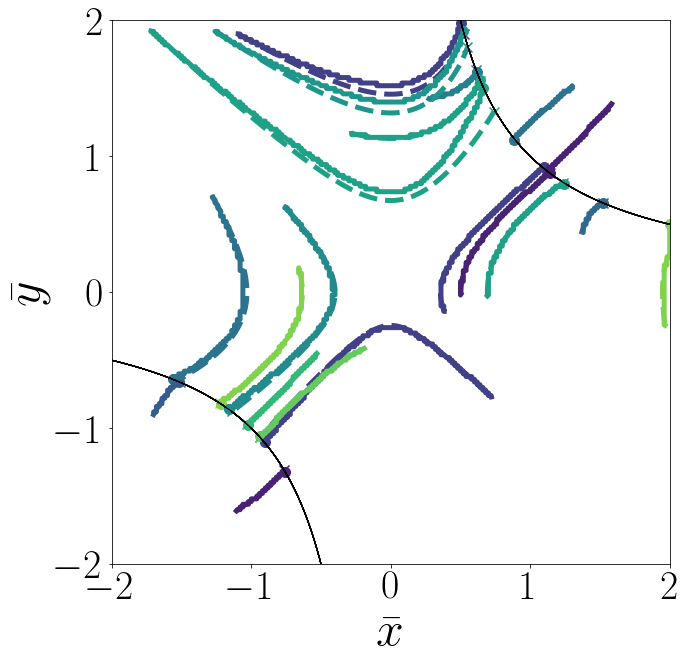}
            \subcaption{$m(x,y) = (1-x\,y)^2.$ (the black line represents the hyperbola $x\,y = 1$).}
            \label{fig:1_xy_squared}
        \end{subfigure}
    \caption{\textbf{Simulated trajectories of the dominant alleles $\bar{x}$ and $\bar{y}$.} The plain lines correspond {\color{Black}to the trajectories of $(x^{\max}(t),y^{\max}(t))$ that realizes the maximum of $n_\varepsilon(t,\cdot,\cdot)$ in the {\color{black}numerical resolution} of \eqref{P_n}}, with $\varepsilon = 0.01$. The dashed lines correspond to the numerical {\color{black}resolution} of the canonical equations given in \ref{prop:monomorphism}. Each color corresponds to one trajectory starting at an initial monomorphic state chosen randomly and uniformly in $[-2,2]^2$. The final states of {\color{black}the trajectories obtained from the discretization} of \eqref{P_n} are indicated by full circles. This figure shows that the qualitative derivation of the section are numerically consistent with the model.}
    \label{fig:domninant_alleles_trajectories}
\end{figure}
\section{Preliminary results on the well-posedness of \ref{P_n}: proof of \cref{thm:well_posedness}}
\label{sec:well_posedness}
{\color{black}In this section, we prove \cref{thm:well_posedness}}.
\paragraph{1. Well-posedness of \ref{P_n}.}  \bl{The proof of this part relies on two lemmas. The first one establishes a priori estimates on solutions of \eqref{P_n} on $[0,T[$ for $T>0$. The second one provides the Lipschitz property which enables us to apply the local Cauchy-Lipschitz theorem to show the existence and uniqueness of a maximal solution on a open subset of $[0,T[$. Finally, we show that this maximal solution is global on $[0,T[$ thanks to the estimates of the first lemma, and extend it on $\R_+$.}
\begin{lem}
Let $T>0$ and $n_\varepsilon\in C^1([0,T]\times I\times J)$ be a solution of \eqref{P_n}. Then, {\color{black}under the assumptions \ref{hyp:m} and \ref{hyp:rho0unifbound}}, we have the following a priori bounds for $t\in[0,T]$
\begin{equation}
\label{eq:bornes_rho}
     \rho_0^- \leq \rho_\varepsilon(t) \leq \rho_0^+, \quad \|\rho^X_\varepsilon(t,\cdot)\|_\infty \leq \|\rho^{X,0}_\varepsilon\|_\infty\,e^{\frac{r\,t}{\varepsilon}},\quad \|\rho^Y_\varepsilon(t,\cdot)\|_\infty \leq \|\rho^{Y,0}_\varepsilon\|_\infty\,e^{\frac{r\,t}{\varepsilon}},
\end{equation}
\begin{equation}
\label{eq:borne_n}
   \begin{aligned}
          \|n_\varepsilon(t,\cdot,\cdot)\|_\infty 
         &\leq N_{T}:=\e^{\frac{rT}{2\varepsilon}}\|n_\varepsilon(0,\cdot,\cdot)\|_\infty +e^{\frac{2rT}{\varepsilon}}\frac{\|\rho^{X,0}_\varepsilon\|_\infty\,\|\rho^{Y,0}_\varepsilon\|_\infty}{\rho_0^-}(e^\frac{rT}{2\varepsilon}-1),
    \end{aligned} 
\end{equation}
\begin{equation}
\label{eq:borne_dxn}
    \begin{aligned}
         \|\partial_x n_\varepsilon(t,\cdot,\cdot)\|_\infty \leq N^x_{T}:=\left(\|\partial_x n_\varepsilon^0\|_\infty+N_{T}\,\frac{T}{\varepsilon}\left\|\partial_x m\right\|_\infty\right)\,e^{\frac{\left\|\rho_\varepsilon^{X,0}\right\|_\infty\,|J|}{2\rho_0^-}e^{\frac{r\,T}{\varepsilon}}+\frac{rT}{2\varepsilon}},
    \end{aligned}
\end{equation}
\begin{equation}
\label{eq:borne_dyn}
    \begin{aligned}
         \|\partial_y n_\varepsilon(t,\cdot,\cdot)\|_\infty \leq N^y_{T}:=\left(\|\partial_y n_\varepsilon^0\|_\infty+N_{T}\,\frac{T}{\varepsilon}\left\|\partial_y m\right\|_\infty\right)\,e^{\frac{\left\|\rho_\varepsilon^{Y,0}\right\|_\infty\,|I|}{2\rho_0^-}e^{\frac{r\,T}{\varepsilon}}+\frac{rT}{2\varepsilon}}.
    \end{aligned}
\end{equation}
\label{lem:TM}

\label{lem:TM}
\end{lem}

{\color{black}The proof of \cref{lem:TM} relies on classical {\color{black}computations} and is left to be consulted in \cref{app:proof_lem_TM}.} 

\bl{Let $T>0$. Before stating the next lemma, let us define, for $n \in C^1(I\times J,\, \R_+)$
\begin{equation*}
    \rho(n):=\iint_{I\times J} n(x,y)\,dx\,dy,\quad \rho^X(n) (y) := \int_{I} n(x,y)\,dx, \quad \rho^Y(n) (x) := \int_{J} n(x,y)\,dy.
\end{equation*}
We also define
\begin{align*}
    \mathcal{A}^0 = & \left\{n \in \bl{C^1(I\times J,\, \R_+)}\,|\, \rho(n)\in[\rho_0^-,\rho_0^+]\right\},
\end{align*}
which is a closed subspace of $ \mathcal{X}:= \bl{C^1(I\times J,\, \R_+)}$ and has a Banach space structure with the norm $\|.\|_\mathcal{X} := \|.\|_\infty +  \|\partial_x\cdot\|_\infty+\|\partial_y\cdot\|_\infty $. Finally, let us also define
\begin{equation}
f : 
\begin{cases}
\begin{aligned}
 &\mathcal{A}^0  \rightarrow   \mathcal{X} \\
 &n  \mapsto  f(n) : (x,y) \mapsto \frac{r}{2}\frac{\rho^X(n)(y) \, \rho^Y(n)(x)}{\rho(n)} - \left[m(x,y)+ \kappa \rho(n) - \frac{r}{2}\right]n(x,y).
\end{aligned}
\end{cases}
\label{eq:f}
\end{equation}
Next, we state the following lemma, whose proof can be found in \cref{app:proof_lem_contraction}.
\begin{lem}
Under the assumption \ref{hyp:m}, $f$ is locally Lipschitz on $\mathcal{A}^0$.
\label{lem:contraction}
\end{lem}

Finally, the definition of $f$ in \eqref{eq:f} implies that \ref{P_n} can be reformulated as the following Cauchy problem:

\begin{equation*}
\tag{$P\left(n_\varepsilon\right)$}
    \begin{aligned}
    \begin{cases}
    \partial_tn_\varepsilon(t,x,y) = \frac{1}{\varepsilon} f(n(t,\cdot,\cdot))(x,y),\qquad  (t,x,y) \in \R_+\times I\times J,\\
    \\
    n_\varepsilon(0,x,y) = n^0_\varepsilon(x,y).
    \end{cases}
    \end{aligned}
\end{equation*}
For $T>0$, under the assumption \ref{hyp:rho0unifbound}, and thanks to \cref{lem:TM} and \cref{lem:contraction}, we apply the local Cauchy-Lipschitz theorem to show that there exists a unique maximal solution $n_\varepsilon \in C^1(\Omega\times I \times J)$ to \ref{P_n}, where $\Omega$ is an open subset of $[0,T[$. Next, the estimates on $\left\|n_\varepsilon(t,\cdot,\cdot)\right\|_\mathcal{X}$ stated in \cref{lem:TM} and uniform for $t\in[0,T[$ imply that $\Omega=[0,T[$, which means that the maximal solution is global on $[0,T[$. As the latter holds for any $T>0$, we deduce that that there exists a unique solution $n_\varepsilon \in C^1(\R_+\times I \times J)$ to \ref{P_n}}.

~\paragraph{2. Weak convergence of $(n_\varepsilon)$.}
From \cref{lem:TM}, for $\varepsilon>0$ and $t\in[0,T]$, {\color{black}we have}:
\[\left\|n_\varepsilon(t,\cdot,\cdot)\right\|_{L^1(I\times J)} \leq \rho^+_0.\]
Consequently, $(n_\varepsilon)$ is bounded in $L^\infty\left([0,T],L^1(I\times J)\right)$. Hence, there exists a subsequence that converges in $L^\infty\left(w^*-[0,T],M(I\times J)\right)$ to a measure $n$.

\section{Proof of \ref{prop:nu} and regularity estimates on $u_\varepsilon$}
\label{sec:lipschitz}
In this section, we provide the proofs of the {\color{black}regularity estimates that will be used {\color{black}in the proof of} \cref{thm:convergence_u}.}

\subsection{Proof of \ref{prop:nu}}
In this subsection, we prove the {\color{black}\ref{prop:nu}, which is a key step to prove the additive separation of variables for} $u$ (see \cref{thm:convergence_u}).

Let $t\in [0,T]$, $(x,y) \in I\times J$. We differentiate \bl{$\nu_\varepsilon=\frac{\rho^X_\varepsilon\,\rho^Y_\varepsilon}{n_\varepsilon\rho_\varepsilon}$} with regard to $t$ to find:
\begin{equation}
\label{eq:calcul_dt_nu}
    \begin{aligned}
    \varepsilon\,\partial_t\nu_\varepsilon (t,x,y) &= \frac{\rho^X_\varepsilon\,\varepsilon\,\partial_t\rho^Y_\varepsilon +\rho^Y_\varepsilon\,\varepsilon\,\partial_t\rho^X_\varepsilon}{n_\varepsilon\,\rho_\varepsilon} -\left(\frac{\varepsilon\,\partial_t n_\varepsilon}{n_\varepsilon}+\frac{\varepsilon\,\partial_t \rho_\varepsilon}{\rho_\varepsilon}\right)\nu_\varepsilon\\
    &= 2\,(r-\kappa\,\rho_\varepsilon)\,\nu_\varepsilon - {\color{black}\int_I m(x',y)\,\frac{n_\varepsilon(x',y)\,\rho_\varepsilon^Y(x)}{n_\varepsilon(x,y)\rho_\varepsilon}\,dx'-\int_Jm(x,y')\,\frac{n_\varepsilon(x,y'),\rho_\varepsilon^X(y)}{n_\varepsilon(x,y)\rho_\varepsilon}\,dy'}\\
    &- \nu_\varepsilon\left(\frac{r}{2}\nu_\varepsilon+\frac{r}{2} - m(x,y) -\kappa\,\rho_\varepsilon + r-\kappa\,\rho_\varepsilon - \iint_{I\times J}m(x,y)\,\frac{n_\varepsilon(x,y)}{\rho_\varepsilon}dx\,dy\right)\\
    & = \frac{r}{2}\nu_\varepsilon(1-\nu_\varepsilon) + \nu_\varepsilon\left(m(x,y)+\iint_{I\times J}m(x,y)\,\frac{n_\varepsilon(x,y)}{\rho_\varepsilon}dx\,dy\right)\\&- \iint_{I\times J}\left(m(x',y)+m(x,y')\right)\frac{n_\varepsilon(x',y)\,n_\varepsilon(x,y')}{n_\varepsilon(x,y)\,\rho_\varepsilon}\,dx'\,dy'.
\end{aligned}
\end{equation}
Since $m\geq0$ and $n_\varepsilon,\rho^X_\varepsilon,\rho^Y_\varepsilon,\rho_\varepsilon>0$, we get:
\begin{equation*}
     \varepsilon\,\partial_t\nu_\varepsilon (t,x,y)\leq \left(\frac{r}{2}+2\,\|m\|_\infty\right)\,\nu_\varepsilon - \frac{r}{2}\,\nu_\varepsilon^2.
\end{equation*}
Hence:
\begin{align*}
    \nu_\varepsilon(t,x,y) &\leq \frac{1}{\frac{1}{\nu^0_\varepsilon(x,y)}\,e^{- \left(\frac{r}{2}+2\|m\|_\infty\right)\,\frac{t}{\varepsilon}}+\frac{r}{\left(r+4\|m\|_\infty\right)}\left(1-e^{- \left(\frac{r}{2}+2\|m\|_\infty\right)\,\frac{t}{\varepsilon}}\right)}\\
    &\leq \frac{1}{\min\left(\frac{1}{\nu^0_\varepsilon(x,y)},\frac{r}{\left(r+4\|m\|_\infty\right)} \right)} \leq \max\left(\|\nu_\varepsilon^0\|_\infty,\, \left(1+\frac{4\|m\|_\infty}{r}\right)\right)\,\leq \nu_M.
\end{align*} Similarly, from \eqref{eq:calcul_dt_nu}, we have:
\[\varepsilon\,\partial_t\nu_\varepsilon (t,x,y)\geq \left(\frac{r}{2}-2\|m\|_\infty \right)\nu_\varepsilon - \frac{r}{2}\nu_\varepsilon^2.\]
Recall from \ref{hyp:m} that: $r>2\|m\|_\infty $. Hence:
\begin{align*}
    \nu_\varepsilon(t,x,y) &\geq \frac{1}{\frac{1}{\nu^0_\varepsilon(x,y)}\,e^{- \left(\frac{r}{2}-2\,\|m\|_\infty\right)\,\frac{t}{\varepsilon}}+\frac{r}{ r-4\,\|m\|_\infty}\left(1-e^{- \left(\frac{r}{2}-2\,\|m\|_\infty\right)\,\frac{t}{\varepsilon}}\right)}\\
    &\geq \frac{1}{\max\left(\frac{1}{\nu^0_\varepsilon(x,y)},\frac{r}{r-4\|m\|_\infty} \right)} \geq \min\left(\|\nu_\varepsilon^0\|_\infty,\, \left(1
    -\frac{4\|m\|_\infty}{r}\right)\right) = \nu_m.
\end{align*}

\subsection{Regularity estimates on $u_\varepsilon$}
In this subsection, we prove the regularity estimates that underlie the convergence {\color{black}of $u_\varepsilon$ based on the Arzela-Ascoli theorem}.
\begin{prop}
\label{prop:uniform_estimates_u}
{\color{black}Assume that \ref{hyp:m}, \ref{hyp:u0W1infty} and \ref{hyp:rho0unifbound} hold.} Let $\varepsilon>0$, $T>0$, and $u_\varepsilon \in C^1([0,T]\times I\times J)$ be the solution of \ref{P_u}. Then, $u_\varepsilon$ is Lipschitz continuous in time and in space,  and is bounded in $C([0,T]\times I\times J)$, all the bounds being uniform with regard to $\varepsilon$.
\end{prop}
\begin{proof}[Proof of \ref{prop:uniform_estimates_u}.]
~\paragraph{Lipschitz bounds in time.}

\bl{From \ref{hyp:m} and \cref{thm:well_posedness}, $\nu_\varepsilon$ is the only term in \ref{P_u} whose boundness is not a priori ensured. However,} \ref{prop:nu} provides an upper bound for $\nu_\varepsilon$ which implies directly the following uniform Lipschitz bound in time on $u_\varepsilon$:
\[\|\partial_t u_\varepsilon\|_\infty \leq \|m\|_\infty +\kappa\,\rho_0^+ + {\color{black}\frac{r}{2}\left(1+\nu_M\right)}.\]

~\paragraph{Lipschitz bounds in space.}

In this paragraph, we rely on a maximum principle to show the following inequalities for all $(t,x,y) \in [0,T]\times I\times J$:
\begin{equation}
\label{eq:lip_space}
    \left|\partial_x u_\varepsilon(t,x,y)\right| <2\|\partial_x m\|_\infty T +\|\partial_xu^0_\varepsilon\|_\infty + 1, \quad \left|\partial_y u_\varepsilon(t,x,y)\right| <2\|\partial_y m\|_\infty T +\|\partial_y u^0_\varepsilon\|_\infty + 1.
\end{equation}
The latter together with \ref{hyp:u0W1infty} implies that $(u_\varepsilon)$ is uniformly Lipschitz continuous in space.

Let us show \eqref{eq:lip_space}.
For $t\in [0,T], (x,y,x',y') \in \mathring{I}^4$, define $\Delta_\varepsilon (x',y',x,y,t) = u_\varepsilon(t,x',y)+u_\varepsilon(t,x,y') - u_\varepsilon(t,x,y)$. Differentiating the equation on $u_\varepsilon$ from \ref{P_u} with regard to $x${\color{black}, we obtain}:
\[\partial_t \,\partial_x u_\varepsilon = -\partial_x m + \frac{r}{2\rho_\varepsilon}\displaystyle \iint_{I\times J}\frac{1}{\varepsilon}\left[\partial_x u_\varepsilon (x,y') - \partial_x u_\varepsilon(x,y)\right] e^{\Delta_\varepsilon (x',y',x,y,t)}dx'\,dy'.\]
Let us define for $(x,y) \in I\times J$:
\[w_\varepsilon(t,x,y) = \partial_x u_\varepsilon(t,x,y) - 2 \|\partial_x m\|_\infty \,t - \|\partial_xu^0_\varepsilon\|_\infty - 1.\]
First, we have that for all $(x,y) \in I\times J$: $w_\varepsilon(0,x,y) <0$. Next, let us assume that there exists $t>0$ such that $\underset{I\times J}{\max}\; w_\varepsilon(t,\cdot) \geq 0$. Then we can define:
\[t_0 = \inf \{t>0, \underset{I\times J}{\max}\; w_\varepsilon(t,\cdot)\geq 0\}.\]
By continuity of $\partial_x u_\varepsilon$ {\color{black}at $t=0$} and compactness of $I$, we have: $t_0>0$. Let $(x_0,y_0) \in I\times J$ be such that: $w_\varepsilon(t_0,x_0,y_0) = \underset{I\times J}{\max}\; w_\varepsilon(t_0,\cdot)$. Then, we have:
\begin{align*}
    &0\leq \partial_t\,w_\varepsilon(t_0,x_0,y_0)\\
    &= - \partial_x m(x_0,y_0) + \frac{r}{2\rho_\varepsilon(t_0)}\displaystyle \iint_{I\times J}\frac{1}{\varepsilon}\left[w_\varepsilon (x_0,y') - w_\varepsilon(x_0,y_0)\right] e^{\Delta_\varepsilon (x',y',x,y,t)}dx'\,dy - 2\|\partial_x m\|_\infty\\
 &\leq -\|\partial_x m\|_\infty <0.
\end{align*}
which is a contradiction. Therefore:
\[\forall t\in[0,t],\,(x,y)\in I\times J, w_\varepsilon(t,x,y) < 0,\] which yields:
\[\partial_x u_\varepsilon(t,x,y) <2\|\partial_x m\|_\infty t +\|\partial_xu^0_\varepsilon\|_\infty + 1. \]
Next, let us consider, for $(t,x,y) \in [0,T]\times I\times J$:
\[v_\varepsilon(t,x,y) = \partial_xu_\varepsilon(t,x,y) + 2 \|\partial_x m\|_\infty\,t + \|\partial_x u^0_\varepsilon\|+1.\]
We can repeat the argument above switching maximum to minimum. First, we have that $v_\varepsilon(0,\cdot,\cdot) >0$. If we assume that there exists $t>0$ such that $\min v_\varepsilon(t,\cdot,\cdot) \leq 0$ and define:
\[t_0 = \inf\{t>0, \min v_\varepsilon(t,\cdot,\cdot) \leq 0\} >0,\]
and $(x_0,y_0)$ realising that minimum, we would have:
\begin{align*}
    &0\geq \partial_t v_\varepsilon(t_0,x_0,y_0)\\
    &= - \partial_x m(x_0,y_0) + \frac{r}{2\rho_\varepsilon(t_0)}\displaystyle \iint_{I\times J}\frac{1}{\varepsilon}\left[w_\varepsilon (x_0,y') - w_\varepsilon(x_0,y_0)\right] e^{\Delta_\varepsilon (x',y',x,y,t)}dx'\,dy + 2\|\partial_x m\|_\infty\\
    &\geq \|\partial_x m\|_\infty >0.
\end{align*}
Which is a contradiction. Thus $v_\varepsilon >0$ and for all $(t,x,y) \in [0,T]\times I\times J$:
\[\partial_x u_\varepsilon > -2\|\partial_x m\|_\infty \,t - \|\partial_x u^0_\varepsilon\|-1.\]
The bound on $\partial_y u_\varepsilon$ can be obtained using similar arguments.

~\paragraph{Uniform $L^\infty$ bounds on $u_\varepsilon$.}
Let us show the following lemma:
\begin{lem}
For any $\delta >0$, there exists $\varepsilon_0 >0$ such that for all $0<\varepsilon\leq \varepsilon_0$:
\[-\delta < \max u_\varepsilon <\delta,\qquad \min u_\varepsilon > -\delta - |I|\,\left(\|\partial_x u_\varepsilon\|_\infty+\|\partial_y u_\varepsilon\|_\infty\right).\]
Hence, $(u_\varepsilon)$ is uniformly bounded for $\varepsilon$ small.
\label{lem:bounds_u}
\end{lem}
\begin{proof}
~\paragraph{1. Bounds on $\max u_\varepsilon$.}
Let $\delta>0$. On the one hand, we have:
\[\rho^-_0 \leq \iint_{I\times J} \frac{\exp\left(\frac{u_\varepsilon(x,y)}{\varepsilon}\right)}{\varepsilon}\,dx\,dy \leq |I|^2\,\frac{\exp\left(\frac{\max u_\varepsilon}{\varepsilon}\right)}{\varepsilon},\]
which leads to:
\[\max u_\varepsilon \geq \varepsilon\,\log\left(\frac{\varepsilon\,\rho^-_0}{|I|^2}\right) \underset{\varepsilon\rightarrow 0}{\longrightarrow} 0.\]
That implies that there exists $\varepsilon_0>0$, such that:
\[\forall \,0<\varepsilon\leq \varepsilon_0, \quad -\delta < \max u_\varepsilon.\]
On the other hand, if $\max u_\varepsilon = u_\varepsilon(x_m,y_m) >0$, then, for all $(x,y)\in I\times J$, we have:
\begin{equation}
\label{eq:lower_bound_u}
    u_\varepsilon (x,y) \geq u_\varepsilon(x_m,y_m) - \left\|\partial_x u_\varepsilon\right\|_\infty |x-x_m| - \left\|\partial_y u_\varepsilon\right\|_\infty |y-y_m|.
\end{equation}
Therefore, using the fact that $u_\varepsilon$ is Lipschitz continuous in space, we obtain, for $(x,y) \in I\times J$ such that $|x-x_m| \leq \frac{u_\varepsilon(x_m,y_m)}{4\left\|\partial_x u_\varepsilon\right\|_\infty},\; |y-y_m| \leq \frac{u_\varepsilon(x_m,y_m)}{4\left\|\partial_y u_\varepsilon\right\|_\infty}$:
\[u_\varepsilon(x,y) \geq  \frac{u_\varepsilon(x_m,y_m)}{2}.\]
We deduce that:
\begin{align*}
    \rho^+_0 \geq \iint_{I\times J} \frac{\exp\left(\frac{u_\varepsilon(x,y)}{\varepsilon}\right)}{\varepsilon}\,dx\,dy &\geq
    \iint_{|x-x_m| \leq \frac{u_\varepsilon(x_m,y_m)}{4\left\|\partial_x u_\varepsilon\right\|_\infty},\; |y-y_m| \leq \frac{u_\varepsilon(x_m,y_m)}{4\left\|\partial_y u_\varepsilon\right\|_\infty}} \frac{\exp\left(\frac{u_\varepsilon(x,y)}{\varepsilon}\right)}{\varepsilon}\,dx\,dy\\&\geq \frac{u_\varepsilon(x_m,y_m)^2}{4\,\left\|\partial_x u_\varepsilon\right\|_\infty\, \left\|\partial_y u_\varepsilon\right\|_\infty}\,\frac{\exp\left(\frac{u_\varepsilon(x_m,y_m)}{2\varepsilon}\right)}{\varepsilon}.
\end{align*}
The latter yields that if $u_\varepsilon(x_m,y_m) \geq \delta$, then:
\[\rho^+_0 \geq \frac{\delta^2}{4\,\left\|\partial_x u_\varepsilon\right\|_\infty\, \left\|\partial_y u_\varepsilon\right\|_\infty}\,\frac{\exp\left(\frac{\delta}{2\varepsilon}\right)}{\varepsilon}\underset{\varepsilon\rightarrow 0}{\longrightarrow} +\infty.\]
Therefore, there exists $\varepsilon_0>0$ such that:
\[\forall \,0<\varepsilon\leq \varepsilon_0,\quad -\delta < \max u_\varepsilon < \delta.\]

\paragraph{2. Bound on $\min u_\varepsilon$.}
From \eqref{eq:lower_bound_u}, for all $(x,y) \in I\times J$, we have:
\[u_\varepsilon(x,y) > \max u_\varepsilon - |I|\left(\left\|\partial_x u_\varepsilon\right\|_\infty+\left\|\partial_y u_\varepsilon\right\|_\infty\right) > -\delta- |I|\left(\left\|\partial_x u_\varepsilon\right\|_\infty+\left\|\partial_y u_\varepsilon\right\|_\infty\right).\]
Thanks to \eqref{eq:lip_space}, the r.h.s is uniformly bounded.
\end{proof}
\end{proof}

\section{Proof of \cref{thm:convergence_u}}
\label{sec:u}
In this section, we provide the proof for the main result of this paper, which is the convergence of $u_\varepsilon$ \bl{toward} a non-positive limit $u$ that separates additively the variables. We also link the support of $n$ to the zeros of $u$ and derive a limit equation.

\paragraph{Limit $\boldsymbol{u}$.}
From \ref{prop:uniform_estimates_u}, there exists $\varepsilon_0>0$ such that $(u_\varepsilon)_{\varepsilon\leq \varepsilon_0}$ is uniformly bounded in $C^0([0,T]\times I\times J)$, and uniformly Lipschitz continuous in space and time. Hence, from the theorem of Arzela-Ascoli, after extraction of a subsequence, $(u_\varepsilon)$ converges uniformly toward a limit $u\in C^0([0,T]\times I\times J)$, that is also Lipschitz continuous.

\paragraph{$\boldsymbol{u(t,x,y) \leq \max u(t,x,\cdot) + \max u(t,\cdot,y)}$.}
From \ref{hyp:nu0} and \ref{prop:nu}, there exists $\nu_m>0$ such that:
\begin{align*}
    \forall (t,x,y) \in [0,T]\times I\times J, \quad \nu_m &\leq \nu_\varepsilon(t,x,y)\\ 
    &= \displaystyle\iint_{I\times J} \frac{1}{\varepsilon\,\rho_\varepsilon}\,\exp\left[{\frac{u_\varepsilon(t,x,y')+u_\varepsilon(t,x',y)-u_\varepsilon(t,x,y)}{\varepsilon}}\right]\bl{dx'\,dy'}\\
    &\leq \frac{|I|^2}{\rho^-_0} \frac{1}{\varepsilon}\,\exp\left[{\frac{\max(u_\varepsilon(t,x,\cdot))+\max u_\varepsilon(t,\cdot,y))-u_\varepsilon(t,x,y)}{\varepsilon}}\right].\\
\end{align*}
Moreover, for all $(t,x,y)\in [0,T]\times I\times J$, and $\delta>0$, there exists $\varepsilon_0>0$ such that for $0<\varepsilon\leq \varepsilon_0$, 
\[\max(u_\varepsilon(t,x,\cdot)) \leq \max(u(t,x,\cdot))+\delta,\quad\max(u_\varepsilon(t,\cdot,y)) \leq \max(u(t,\cdot,y))+\delta.\]
We deduce that, for $0<\varepsilon<\varepsilon_0$:
\begin{align*}
    u_\varepsilon(t,x,y)-u(t,x,y) + \varepsilon\log\left(\varepsilon\,\frac{\rho^-_0\,\nu_m}{|I|^2}\right) -2\delta\leq \max(u(t,x,\cdot))+\max(u(t,\cdot,y)) - u(t,x,y).
\end{align*}
Letting $\delta$ and $\varepsilon$ vanish yields:
\[u(t,x,y)\leq \max(u(t,x,\cdot))+\max(u(t,\cdot,y)).\]

\paragraph{$\boldsymbol{u(t,x,y) \geq \max u(t,x,\cdot) + \max u(t,\cdot,y)}$.}
For $\delta>0$, there exists $\varepsilon_0>0$ such that for $0<\varepsilon\leq \varepsilon_0$, for $(t,x,y)\in [0,T]\times I\times J$,
\[\max(u_\varepsilon(t,x,\cdot)) \geq \max(u(t,x,\cdot))-\delta,\quad\max(u_\varepsilon(t,\cdot,y)) \geq \max(u(t,\cdot,y))-\delta.\]
Let $\varepsilon \leq \varepsilon_0$ and $y_\varepsilon(x)$ be such that: $u_\varepsilon(t,x,y_\varepsilon(x))= \max(u_\varepsilon(t,x,\cdot))$. Since $u_\varepsilon$ is uniformly Lipschitz in space (\ref{prop:uniform_estimates_u}), we can choose $M>0$ such that:
\begin{align*}
    \forall (y,y')\in I\times J,\; |u_\varepsilon(t,x,y)-u_\varepsilon(t,x,y')| \leq M |y-y'|.
\end{align*}
Combining the last two estimations leads to:
\[|y-y_\varepsilon(x)|\leq \frac{\delta}{M} \implies u_\varepsilon(t,x,y') > \max u(t,x,\cdot) - 2\delta.\]
The same holds for $\max(u(t,\cdot,y))$. Hence, from \ref{prop:nu}, there exists $\nu_M$ such that:
\begin{align*}
    \rho_0^+\,\nu_M &\geq \displaystyle\iint_{I\times J} \frac{1}{\varepsilon}\,\exp\left[{\frac{u_\varepsilon(t,x,y')+u_\varepsilon(t,x',y)-u_\varepsilon(t,x,y)}{\varepsilon}}\right]\bl{dx'\,dy'}\\
    &\geq \left(\frac{\delta}{M}\right)^2\frac{1}{\varepsilon}\exp\left[{\frac{\max(u(t,x,\cdot))+\max(u(t,\cdot,y))-4\delta-u_\varepsilon(t,x,y)}{\varepsilon}}\right].
\end{align*}
We next obtain:
\begin{align*}
     \varepsilon\,\log\left(\varepsilon\,\frac{\rho^+_0\nu_M\,M^2}{\delta^2}\right)+4\delta+u_\varepsilon(t,x,y) - u(t,x,y)\geq \max(u(t,x,\cdot))+\max(u(t,\cdot,y))-u(t,x,y).
\end{align*}
Letting $\delta$ and $\varepsilon$ vanish yields:
\[u(t,x,y)\geq \max(u(t,x,\cdot))+\max(u(t,\cdot,y)).\]
{\color{black}This concludes the proof of \eqref{eq:u_additive}.}
\paragraph{$u$ is non-positive.}
This property follows directly {\color{black}from} the uniform convergence of $u_\varepsilon$ \bl{toward} $u$ and the \bl{uniform estimates on $\max(u_\varepsilon)$} from \cref{lem:bounds_u}.

~\paragraph{Support of $n$ and zeros of $u$.}
Let $t\in [0,T]$. Let: $(x_0,y_0) \notin \left\{(x,y)\,|\,u(t,x,y)=0\right\}$. Since $u(t,\cdot,\cdot)$ is \bl{uniformly continuous}, there exists $\delta>0$ such that: $\max\left(|x'-x_0|,|y'-y_0|\right)\leq \delta \implies u(t,x',y') \leq \frac{u(t,x_0,y_0)}{2} \bl{<0}$. Also, \bl{thanks to the strong convergence of $(u_\varepsilon)$ toward $u$}, there exists $\varepsilon_0>0$ such that, for all $\varepsilon<\varepsilon_0$, we have: $\|u_\varepsilon-u\|_\infty \leq \frac{\left|u(t,x_0,y_0)\right|}{4}$. Then, for $\varepsilon<\varepsilon_0$, we have:
\begin{align*}
    \int_{[x_0-\delta,x_0+\delta]\times [y_0-\delta,y_0+\delta]}n_\varepsilon(t,x',y')\,dx'\,dy' &=\int_{[x_0-\delta,x_0+\delta]\times  [y_0-\delta,y_0+\delta]}\frac{e^{\frac{u_\varepsilon(t,x',y')}{\varepsilon}}}{\varepsilon}\,dx'\,dy' \\
    &=\bl{\int_{[x_0-\delta,x_0+\delta]\times  [y_0-\delta,y_0+\delta]}\frac{e^{\frac{u(t,x',y') + \frac{\left|u(t,x_0,y_0)\right|}{4}}{\varepsilon}}}{\varepsilon}\,dx'\,dy'} \\
    &\leq 4\delta^2 \,\frac{e^{\frac{u(t,x_0,y_0)}{4\varepsilon}}}{\varepsilon} \\
    & \underset{\varepsilon \to 0}{\longrightarrow} 0.
\end{align*}
From the weak convergence result of \cref{thm:well_posedness}, $(x_0,y_0) \notin \operatorname{supp}(n(t,\cdot,\cdot))$.
~\paragraph{Limit equation on $u^X$.}
Let $\varepsilon>0$. {\color{black}From the equation \eqref{eq:rhoX_rhoY} verified by $\rho^X_\varepsilon$}, we get, by integration:
\begin{equation}\forall (t,y) \in [0,T]\times {\color{black}J},\quad
\varepsilon\,\log\left(\frac{\rho^X_\varepsilon(t,y)}{\rho^X_  {\varepsilon}(0,y)}\right) = rt- \kappa\int_0^t\rho_\varepsilon(s)ds -\int_0^t \int_I m(x,y)\,\frac{n_\varepsilon(s,x,y)}{\rho^X_\varepsilon(s,y)}dx\,ds.
    \label{eq:eq_rho_X}
\end{equation}
Let us define $\phi^X_\varepsilon \in C\left([0,T]\times I\times J\right)$ by:
\[\phi^X_\varepsilon(t,x,y) = \frac{n_\varepsilon(t,x,y)}{\rho^X_\varepsilon(t,y)}.\]
\paragraph{1. Convergence of $\phi^X_\varepsilon$ to $\phi^X$.}
For $(t,y) \in [0,T]\times {\color{black}J}$, we have:
\[\int_I \phi^X_\varepsilon(t,x',y)\,dx' = 1.\]
Hence, $\left(\phi^X_\varepsilon\right)_{\varepsilon>0}$ is bounded {\color{black}in} $L^\infty([0,T]\times {\color{black}J},L^1(I))$. Thus, there exists a subsequence still denoted $\left(\phi^X_\varepsilon\right)_{\varepsilon>0}$  that converges in $L^\infty(w^*-[0,T]\times {\color{black}J},M(I))$ toward a measure $\phi^X$.
\paragraph{Support of $\phi^X(t,\cdot,y)$.}
\ref{prop:nu} implies that, for $f\in C_c(I,\R_+)$, {\color{black}for a.e. $(t,y)$}
\[\frac{1}{\nu_M\,\rho_0^-}\int_I \rho^Y(t,x)\,f(x)\,dx \leq \left\langle \phi^X(t,\cdot,y),\,f\right\rangle \leq \frac{1}{\nu_m\,\rho_0^+}\int_I \rho^Y(t,x)\,f(x)\,dx.\]
Hence, {\color{black}for a.e. $(t,y)$}, $\phi^X(t,\cdot,y)$ and {\color{black}$\rho^Y(t,\cdot)$} share the same support. As $n(t,\cdot,\cdot)$ is supported at $u(t,\cdot,\cdot)^{-1}\left(\left\{(0,0)\right\}\right) = u^X(t,\cdot)^{-1}\left(\left\{0\right\}\right)\times u^Y(t,\cdot)^{-1}\left(\left\{0\right\}\right)$ {\color{black}for a.e. $t$}, we obtain that {\color{black}$\rho^Y(t,\cdot)$} (and therefore $\phi^X(t,\cdot,y)$) is supported at the zeros of $u^Y(t,\cdot)$.
\paragraph{2. $\boldsymbol{\varepsilon\,\log\left(\rho^X_\varepsilon\right)\underset{\varepsilon\to 0}{\longrightarrow} u^X \in C^0([0,T]\times I).}$}

We fix $\delta>0$ and let $\varepsilon_0>0$ be such that: $\forall \varepsilon< \varepsilon_ 0, \,\left\|u_\varepsilon-u\right\|_\infty \leq \delta$. Next, we compute:
\begin{equation}
\label{eq:logrhoX_uX_1}
   \begin{aligned}
    \varepsilon\,\log\left(\rho^X_\varepsilon(t,y)\right) &= \varepsilon \log\left(\int_I \frac{e^{\frac{u_\varepsilon(t,x,y)}{\varepsilon}}}{\varepsilon}\,dx\right) \\
    &\leq \varepsilon\,\log\left(\int_I \frac{e^{\frac{u(t,x,y)+\delta}{\varepsilon}}}{\varepsilon}\,dx\right) \\
    &= \varepsilon\,\log\left(\int_I \frac{e^{\frac{u^X(t,y)+\delta}{\varepsilon}}}{\varepsilon}\,e^{\frac{u^Y(t,x)}{\varepsilon}}\,dx\right)\\
    &\leq u^X(t,y) + \delta -\varepsilon\,\log(\varepsilon) + \varepsilon\,\log\left(\int_I e^{\frac{u^Y(t,x)}{\varepsilon}}\,dx\right)\\
    &\leq u^X(t,y) + \delta -\varepsilon\,\log(\varepsilon) + \varepsilon\,\log\left(|I|\right).
\end{aligned} 
\end{equation}
Similarly, we have:
\begin{equation}
\label{eq:logrhoX_uX_2}
    \varepsilon\,\log\left(\rho^X_\varepsilon(t,y)\right) \geq u^X(t,y) - \delta -\varepsilon\,\log(\varepsilon) +\varepsilon\,\log\left(\int_I e^{\frac{u^Y(t,x)}{\varepsilon}}\,dx\right).
\end{equation}
For all $t\in[0,T]$, we have shown at the step $(ii)$ that there exists $x_0(t) \in I$ such that $u^Y\left(\bl{t}, x_0(t)\right)=0$. We have therefore the following lower bound:
\begin{align*}
    \varepsilon\,\log\left(\int_I e^{\frac{u^Y(t,x)}{\varepsilon}}\,dx\right)&\geq \varepsilon\,\log\left(\int_{x_0(t)-\varepsilon}^{x_0(t)+\varepsilon} e^{\frac{u^Y(t,x)}{\varepsilon}}\,dx\right)\\
    & \geq \varepsilon\,\log\left(\int_{x_0(t)-\varepsilon}^{x_0(t)+\varepsilon} e^{\frac{-M|x-x_0(t)|}{\varepsilon}}\,dx\right)\\
    & \geq -\varepsilon\,M+\varepsilon\,\log\left(2\,\varepsilon\right)\\
\end{align*}
where the intermediate inequality is obtained due to the fact that there exists $M>0$ such that $u(t,\cdot,\cdot)$ is $M$-lipschitz in space, and thus, so is $u^Y(t,\cdot)$.

The two inequalities \eqref{eq:logrhoX_uX_1} and \eqref{eq:logrhoX_uX_2} above ensure the convergence of $\varepsilon\,\log\left(\rho^X_\varepsilon\right)$ toward $u^X$ uniformly {\color{black}in} $[0,T ]\times {\color{black}J}$.
\paragraph{3. Limit equation.}

For all $(t,y) \in [0,T]\times {\color{black}J}$:
\[\int_{0}^t \int_{I} m(x,y) \phi_\varepsilon^X(s,x,y)\,dx\,ds =
\varepsilon\,\log\left(\frac{\rho^X_\varepsilon(t,y)}{\rho^X_  {\varepsilon}(0,y)}\right) - rt + \kappa\int_0^t\rho_\varepsilon(s)ds.\]
From the strong convergence {\color{black}$\varepsilon\,\log\left(\rho^X_\varepsilon\right)\underset{\varepsilon\to 0}{\longrightarrow} u^X \in C^0([0,T]\times J)$} shown previously, the r.h.s of the equality above converges toward a function in $C^0\left([0,T]\times {\color{black}J}\right)$ as $\varepsilon$ vanishes. 
Hence, $G_\varepsilon := (t,y)\mapsto\int_{0}^t \int_{I} m(x,y) \phi_\varepsilon^X(t,x,y)\,dx\,ds$ converges uniformly toward a function denoted $G(t,y)$ in $C^0\left([0,T]\times {\color{black}J}\right)$.

We aim to show that for all $t\in [0,T]$: \[G(t,\cdot) = y\mapsto \int_{0}^t \left\langle \phi^X(t,\cdot,y),m(\cdot,y)\right\rangle\,ds \in L^\infty({\color{black}J}),\]
which would yield \eqref{eq:limit_equation}. Let $f\in L^1({\color{black}J})$. We have, for $t\in[0,T]$:

\begin{align*}
     &\left|\int_{\color{black}J} \left(G(t,y)-\int_{0}^t \left\langle \phi^X(t,\cdot,y),m(\cdot,y)\right\rangle\,ds\right)\,f(y)\,dy\right| \\
     &\leq \left\|G-G_\varepsilon\right\|_\infty\,\left\|f\right\|_1 + \left|\int_0^T\int_{ J}\left\langle\phi^X_\varepsilon(s,\cdot,y) - \phi^X(s,\cdot,y),\,m(\cdot,y)\right\rangle\mathbf{1}_{[0,t]}(s)\,f(y)ds\,dy\right|.
\end{align*}
The first term vanishes because of the uniform convergence of $G_\varepsilon$ to $G$. The second term does the same because of the weak convergence of $\phi^X_\varepsilon$ to $\phi^X$ in $L^\infty(w^*-[0,T]\times {\color{black}J},M(I))$ applied to $(s,x,y) \mapsto \mathbf{1}_{[0,t]}(s)f(y)m(x,y) \in L^1([0,T]\times {\color{black}J},C^0(I))$, since $f \in L^1({\color{black}J})$ and $m\in C^1(I\times J)$. We obtain that for all $t\in[0,T]$, $f\in L^1({\color{black}J})$:

\[\int_{\color{black}J} \left(G(t,y)-\int_{0}^t \left\langle \phi^X(t,\cdot,y), m(\cdot,y)\right\rangle\,ds\right)\,f(y)\,dy= 0.\]
{\color{black}Therefore}, we deduce that {for all $t\in[0,T]$, \color{black}for a.e $y$}, $G(t,y) = \int_{0}^t \left\langle \phi^X(t,\cdot,y),m(\cdot,y)\right\rangle\,ds$.

\section{Convergence of $(\rho_\varepsilon)$ toward a BV function: proof of \cref{thm:BV}}
\label{sec:BV}
In this section, we provide the proof of \cref{thm:BV} under the additional hypothesis that the selection function $m$ is additive \eqref{eq:m_add}.

Let $\varepsilon>0$. First, we have, for {\color{black}all} $t\in[0,T]$:
\begin{align*}
    \int_0^t \left|\frac{d\rho_\varepsilon}{dt}(s)\right|\,ds &= \int_0^t \frac{d\rho_\varepsilon}{dt}(s)\,ds + 2\int_0^t \left(\frac{d\rho_\varepsilon}{dt}\right)_- (s)\,ds\\
    &= \rho_\varepsilon(t) - \rho_\varepsilon(0) + 2\int_0^t \left(\frac{d\rho_\varepsilon}{dt}\right)_- (s)\,ds\\
    &\leq \rho_0^+ + 2\int_0^t \left(\frac{d\rho_\varepsilon}{dt}\right)_- (s)\,ds,
\end{align*}
using the estimates of \cref{lem:TM}. Let us define: \[\mathcal{I}_\varepsilon(t) := \frac{d\rho_\varepsilon}{dt}(t) = \frac{r-\kappa\,\rho_\varepsilon(t)}{\varepsilon}\,\rho_\varepsilon(t) - \iint_{I\times J}m(x,y)\,\frac{n_\varepsilon(t,x,y)}{\varepsilon}\,dx\,dy.\]
To prove that $\rho_\varepsilon$ is locally uniformly bounded in $W^{1,1}([0,T])$, it is sufficient to give an upper bound on $\int_0^t\left(\mathcal{I}_\varepsilon\right)_-\,ds$. To this end, let us notice that {\bl{for a.e. $t$}}:
\begin{equation*}
    \frac{d\left(\mathcal{I}_\varepsilon\right)_-}{dt} (t) = -\mathbf{1}_{\mathcal{I}_\varepsilon\leq0}\,\frac{d\mathcal{I_\varepsilon}}{dt}.
\end{equation*}
We deduce that, \bl{for a.e. $t$}
\begin{equation}
\label{eq:aux_I}
    \begin{aligned}
    \frac{d\left(\mathcal{I}_\varepsilon\right)_-}{dt} (t) &= -\frac{d\mathcal{I_\varepsilon}}{dt}(t)\,\mathbf{1}_{\{\mathcal{I_\varepsilon}(t)\leq 0\}}\\
    &= -\left[\frac{r-2\,\kappa\,\rho_\varepsilon}{\varepsilon}\,\mathcal{I_\varepsilon}(t) -\iint_{I\times J}m(x,y)\,\frac{\partial_t n_\varepsilon(t,x,y)}{\varepsilon}\,dx\,dy\right]\,\mathbf{1}_{\{\mathcal{I_\varepsilon}(t)\leq 0\}}\\
    & = \left[\frac{r-2\,\kappa\,\rho_\varepsilon}{\varepsilon}\,\left(\mathcal{I_\varepsilon}\right)_-(t)\right.\\
    &\left.+\iint_{I\times J}m(x,y)\,\frac{1}{\varepsilon^2}\left(\frac{r}{2}\left[\frac{\rho_\varepsilon^X\,\rho_\varepsilon^Y}{\rho_\varepsilon}+n_\varepsilon\right] - \kappa\,\rho_\varepsilon\,n_\varepsilon -m(x,y)\,n_\varepsilon\right)\,dx\,dy\right]\,\mathbf{1}_{\{\mathcal{I_\varepsilon}(t)\leq 0\}}.
\end{aligned}
\end{equation}
Let us show that the following term is non positive:
\begin{align*}
   &\mathbf{1}_{\{\mathcal{I_\varepsilon}(t)\leq 0\}} \iint_{I\times J}m(x,y)\,\left(\frac{r}{2}\left[\frac{\rho_\varepsilon^X\,\rho_\varepsilon^Y}{\rho_\varepsilon}+n_\varepsilon\right] - \kappa\,\rho_\varepsilon\,n_\varepsilon -m(x,y)\,n_\varepsilon\right)\,dx\,dy \\&=\mathbf{1}_{\{\mathcal{I_\varepsilon}(t)\leq 0\}}\left[ \frac{r}{2}\iint_{I\times J}m(x,y)\frac{\rho_\varepsilon^X\,\rho_\varepsilon^Y}{\rho_\varepsilon} dx\,dy +\left(\frac{r}{2}-\kappa\rho_\varepsilon\right)\iint_{I\times J}m(x,y)\,n_\varepsilon dx dy- \iint_{I\times J}m^2(x,y)\,n_\varepsilon \,dx\,dy\right].
\end{align*}
On the one hand, from the Cauchy-Schwartz inequality, we get:
\[-\rho_\varepsilon\,\iint_{I\times J}m^2(x,y)\,n_\varepsilon\,dx\,dy=-\iint_{I\times J}\sqrt{n_\varepsilon}^2\iint_{I\times J} \left(m(x,y)\sqrt{n_\varepsilon}\right)^2 dx\,dy \leq - \left(\iint_{I\times J} m(x,y)\,n_\varepsilon\,dx\,dy\right)^2.\]
On the other hand, thanks to the additional hypothesis on $m$ \eqref{eq:m_add}, we have
\begin{align*}
    \iint_{I\times J}m(x,y)\frac{\rho_\varepsilon^X\,\rho_\varepsilon^Y}{\rho_\varepsilon}dx\,dy&= \iint_{I\times J}\left[m^X(x)+m^Y(y)\right]\frac{\rho_\varepsilon^X(y)\,\rho_\varepsilon^Y(x)}{\rho_\varepsilon}dx\,dy \\
    &=\int_{I}m^X(x)\rho^Y_\varepsilon(x)\,dx + \int_{I}m^Y(y)\rho^X_\varepsilon(y)\,dy\\
    &=\iint_{I\times J}m^X(x)\,n_\varepsilon(x,y)\,dx\,dy +\iint_{I\times J}m^Y(y)\,n_\varepsilon(x,y)\,dx\,dy\\
    &= \iint_{I\times J}m(x,y)\,n_\varepsilon(x,y)\,dx\,dy.
\end{align*}
We deduce that
\begin{align*}
   &\mathbf{1}_{\{\mathcal{I_\varepsilon}(t)\leq 0\}} \iint_{I\times J}m(x,y)\,\left(\frac{r}{2}\left[\frac{\rho_\varepsilon^X\,\rho_\varepsilon^Y}{\rho_\varepsilon}+n_\varepsilon\right] - \kappa\,\rho_\varepsilon\,n_\varepsilon -m(x,y)\,n_\varepsilon\right)\,dx\,dy\\
   &\leq \mathbf{1}_{\{\mathcal{I_\varepsilon}(t)\leq 0\}} \left[r\iint_{I\times J}m\,n_\varepsilon-\kappa\,\rho_\varepsilon\iint_{I\times J}m\,n_\varepsilon - \frac{1}{\rho_\varepsilon}\left(\iint_{I\times J} m(x,y)\,n_\varepsilon\,dx\,dy\right)^2\right]
   \\&\leq\mathbf{1}_{\{\mathcal{I_\varepsilon}(t)\leq 0\}} \frac{\iint_{I\times J}m\,n_\varepsilon\,dx\,dy}{\rho_\varepsilon}\left[(r-\kappa\,\rho_\varepsilon)\,\rho_\varepsilon - \iint_{I\times J}m\,n_\varepsilon\,dx\,dy\right]\\
   &\leq \mathbf{1}_{\{\mathcal{I_\varepsilon}(t)\leq 0\}} \frac{\iint_{I\times J}m\,n_\varepsilon\,dx\,dy}{\rho_\varepsilon}\,\varepsilon\,\mathcal{I}_\varepsilon(t) \leq 0.
\end{align*}
Consequently, \eqref{eq:aux_I} implies the following inequality
\begin{align*}
    \frac{d\left(\mathcal{I}_\varepsilon\right)_-}{dt} (t) &\leq\frac{r-2\,\kappa\,\rho_\varepsilon}{\varepsilon}\,\left(\mathcal{I_\varepsilon}\right)_-(t)\\
    &\leq \frac{r-2\,\kappa\,\rho_0^-}{\varepsilon}\,\left(\mathcal{I_\varepsilon}\right)_-(t)\\
    &= \frac{2\|m\|_\infty - r}{\varepsilon}\,\left(\mathcal{I_\varepsilon}\right)_-(t).
\end{align*}
Let us define $\delta := r- 2\|m\|_\infty >0$ from \ref{hyp:m}. From the last differential inequality, we deduce that
\begin{equation}
    \left(\mathcal{I}_\varepsilon\right)_- (t) \leq \left(\mathcal{I}_\varepsilon\right)_- (0)\,e^{-\frac{\delta}{\varepsilon}t},
    \label{eq:I_relaxation}
\end{equation}
which concludes the first part of the proof.

For the last part of the theorem, recall that $I_\varepsilon(t) = \frac{d\rho_\varepsilon}{dt}(t)$. Then, the inequality \eqref{eq:I_0} implies that there exists $C>0$ such that
\begin{equation*}
    {\color{black}\left(\mathcal{I}_{\varepsilon}(0)\right)_-} \leq C\,{\color{black}\frac{e^{\frac{o(1)}{\varepsilon}}}{\varepsilon}}.
\end{equation*}
As a corollary of \eqref{eq:I_relaxation}, we obtain that for all $t\in[0,T]$
\begin{equation*}
    \left(\mathcal{I}_\varepsilon\right)_- (t) \leq C\,{\color{black}\frac{e^{\frac{o({\color{black}1})-\delta{\color{black}t}}{\varepsilon}}}{\varepsilon}}.
\end{equation*}
We deduce that the limit $\rho$ is non decreasing. 
\appendix

{\color{black}\section{Proof of \cref{lem:TM}.}
\label{app:proof_lem_TM}
\begin{proof}
[Proof of \cref{lem:TM}]
\item\paragraph{1. Bounds on $\rho_\varepsilon$.}
Integrating \eqref{P_n} leads to $\rho_\varepsilon$ being solution of:
\begin{equation}
    \begin{aligned}
    \begin{cases}
     \varepsilon\, \partial_t \rho_\varepsilon = \left(r - \kappa\,\rho_\varepsilon\right)\rho_\varepsilon - \iint_{I\times J} m(x,y)\,n_\varepsilon(t,x,y) \,dx\,dy\bl{,}\\
     \rho(0) = \rho^0_\varepsilon. 
    \end{cases}
    \end{aligned}
\end{equation}
Since $m, n_\varepsilon\geq 0$, we get that $\rho_\varepsilon$ is a subsolution of the Cauchy problem:
\begin{equation*}
\begin{aligned}
\begin{cases}
\varepsilon \,d_t \tilde{\rho}_\varepsilon = \left(r-\kappa\,\tilde{\rho}_\varepsilon\right)\,\tilde{\rho}_\varepsilon,\\
\tilde{\rho}_\varepsilon(0) = \rho^0_\varepsilon.
\end{cases}
\end{aligned}
\end{equation*}
whose solution is:
\[\forall t\geq 0, \quad \tilde{\rho}_\varepsilon(t) = \frac{1}{\frac{e^{-\frac{rt}{\varepsilon}}}{\rho_\varepsilon^0}+\frac{\kappa}{r}\left(1-e^{-\frac{rt}{\varepsilon}}\right)} \leq \max\left(\rho^0_\varepsilon,\,\frac{r}{\kappa}\right),\]
since $\rho^0_\varepsilon \geq 0$ from assumptions. Using the comparison principle, we obtain 
\[\forall t\geq 0, \quad\rho_\varepsilon(t) \leq \tilde{\rho}_\varepsilon(t)\leq \rho_0^+.\]
Similarly, we get:
\[\forall t \geq 0, \quad\rho_0^-=\leq\min\left(\rho^0_\varepsilon,\,\frac{r-\|m\|_\infty}{\kappa}\right)\leq\rho_\varepsilon(t).\]

\item\paragraph{2. Bounds on $\rho^{X}_\varepsilon,\rho^{Y}_\varepsilon.$}

Integrating \eqref{P_n} {\color{black}with regard to $x$} leads to:
\begin{equation*}
    \begin{aligned}
    \begin{cases}
     \varepsilon\, \partial_t \rho^X_\varepsilon = \left(r - \kappa\,\rho_\varepsilon\right)\rho^X_\varepsilon - \int_{I} m(x,y)\,n_\varepsilon(t,x,y) \,dx\leq r\rho^X_\varepsilon,\\
     \rho^X(0) = \rho^{X,0}_\varepsilon. 
    \end{cases}
    \end{aligned}
\end{equation*}
The upper bound on $\rho^X_\varepsilon$ is then obtained by comparison with $\rho^{X,0}_\varepsilon\,e^{\frac{r\,t}{\varepsilon}}$. The upper bound on $\rho^Y_\varepsilon$ can be proved using similar arguments.

\item\paragraph{3. Bound on $n_\varepsilon$.}

From Duhamel's formula, we obtain, for all $0\leq t\leq T, (x,y)\in I\times J$:
\begin{equation*}
    n_\varepsilon(t,x,y) = n^0_{\varepsilon}(x,y)e^{-\int_0^t\frac{m+\kappa\,\rho_\varepsilon(s)-\frac{r}{2}}{\varepsilon}ds} + \frac{r}{2\varepsilon}\int_0^t \frac{\rho^X_\varepsilon(y,\tau)\rho^Y_\varepsilon(x,\tau)}{\rho_\varepsilon(\tau)}e^{-\int_\tau^t\frac{m+\kappa\,\rho_\varepsilon(s)-\frac{r}{2}}{\varepsilon}ds}d\tau.
\end{equation*}
Hence, using the bounds on $\rho^X_\varepsilon$ and $\rho^Y_\varepsilon$ from \eqref{eq:bornes_rho}, we deduce that:
\begin{equation*}
\begin{aligned}
 \left\|n_\varepsilon(t,\cdot,\cdot)\right\|_\infty &\leq e^{\frac{rt}{2\varepsilon}}\left\|n^0_\varepsilon\right\|_\infty+\frac{\|\rho^{X,0}_\varepsilon\|_\infty\,\|\rho^{Y,0}_\varepsilon\|_\infty}{\rho_0^-}\int_0^t\frac{r}{2\varepsilon} e^{\frac{2r\tau}{\varepsilon}}e^\frac{r(t-\tau)}{2\varepsilon}d\tau \\
 & \leq e^{\frac{rt}{2\varepsilon}}\left\|n^0_\varepsilon\right\|_\infty+\frac{\|\rho^{X,0}_\varepsilon\|_\infty\,\|\rho^{Y,0}_\varepsilon\|_\infty}{\rho_0^-}e^{\frac{2rt}{\varepsilon}}\left(e^\frac{rt}{2\varepsilon}-1\right)
 \\
 & \bl{\leq N_T.}
\end{aligned}
\end{equation*}

\item\paragraph{4. Bound on $\partial_x n_\varepsilon$ and $\partial_y n_\varepsilon$.}

We differentiate \ref{P_n} with respect to $x$ to obtain:
\begin{equation*}
    \begin{aligned}
    \begin{cases}
     \varepsilon\, \partial_t \partial_x n_\varepsilon = \frac{r}{2}\,\left[\frac{\partial_x \rho^Y_{\varepsilon}(x)\,\rho^X_{\varepsilon}(y)}{\rho_\varepsilon}+\partial_x n_\varepsilon(x,y)\right] - \partial_x m(x,y)\,n_\varepsilon - \left(m+\kappa\,\rho_\varepsilon\right)\partial_x n_\varepsilon,\\
     \partial_x n_\varepsilon(0) = \partial_x n^{0}_\varepsilon. 
    \end{cases}
    \end{aligned}
\end{equation*}
Putting the latter under integral form yields:
\[\partial_x n_\varepsilon = \partial_x n^0_\varepsilon\,e^{-\int_0^t\frac{m+\kappa\,\rho_\varepsilon(s)}{\varepsilon}ds}+\frac{1}{\varepsilon} \,\int_0^t \left(\frac{r}{2}\left[\frac{\partial_x \rho^Y_{\varepsilon}(x)\,\rho^X_{\varepsilon}(y)}{\rho_\varepsilon}+\partial_x n_\varepsilon(x,y)\right]-n_\varepsilon\,\partial_x m\right)e^{-\int_\tau^t\frac{m+\kappa\,\rho_\varepsilon(s)}{\varepsilon}ds}d\tau.\]
Hence, by first using the previous bounds \eqref{eq:bornes_rho} {\color{black}and \eqref{eq:borne_n}} and next \bl{using Gronwall's inequality} on $t\mapsto \left\|\partial_x n_\varepsilon(t,\cdot,\cdot)\right\|_\infty$ (second line), we obtain that:
\begin{align*}
    \left\|\partial_x n_\varepsilon(t,\cdot,\cdot)\right\|_{{\color{black}\infty}} &\leq \left\|\partial_x n_\varepsilon^0\right\|_\infty + N_{T_M} \frac{t}{\varepsilon}\left\|\partial_x m\right\|_\infty + \int_0^t\frac{r}{2\varepsilon}\left[\frac{\left\|\rho_\varepsilon^{X,0}\right\|_\infty\,|J|}{\rho_0^-}e^{\frac{r\,\tau}{\varepsilon}}+1\right]\left\|\partial_x n_\varepsilon(\tau,\cdot,\cdot)\right\|_\infty\,d\tau\\
    &\leq \left(\|\partial_x n_\varepsilon^0\|_\infty+N_{T_M}\,\frac{t}{\varepsilon}\left\|\partial_x m\right\|_\infty\right)\,e^{\frac{\left\|\rho_\varepsilon^{X,0}\right\|_\infty\,|J|}{2\rho_0^-}e^{\frac{r\,t}{\varepsilon}}+\frac{rt}{2\varepsilon}}\\
    &\bl{\leq N^X_T}.
\end{align*}
\bl{The bound on $\partial_y n_\varepsilon(t,\cdot,\cdot)$ is obtained similarly.}
\end{proof}}
\section{Proof of \cref{lem:contraction}}
\label{app:proof_lem_contraction}
\bl{\begin{proof}[Proof of \cref{lem:contraction}.]

Let us recall the definition of $\mathcal{A}^0_\varepsilon$
\begin{align*}
    \mathcal{A}^0 = & \left\{n \in \bl{C^1(I\times J,\, \R_+)}\,|\, \rho(n)\in[\rho_0^-,\rho_0^+]\right\},
\end{align*}
and the definition of $f$
\begin{equation*}
f : 
\begin{cases}
\begin{aligned}
 &\mathcal{A}^0  \rightarrow  \mathcal{X} \\
 &n  \mapsto  f(n) : (x,y) \mapsto \frac{r}{2}\frac{\rho^X(n)(y) \, \rho^Y(n)(x)}{\rho(n)} - \left[m(x,y)+ \kappa \rho(n) - \frac{r}{2}\right]n(x,y).
\end{aligned}
\end{cases}
\end{equation*}

From \ref{hyp:m} ($C^1$-regularity of $m$ on $I\times J$) and because bounded linear functionals are Lipschitz, the only terms in the expression of $f$ for which the Lipschitz bound requires additional computations are $\tilde{f}: n \mapsto \left[(x,y) \mapsto \frac{\rho^X(n)(y) \, \rho^Y(n)(x)}{\rho(n)}\right]$ and $\hat{f} : n \mapsto \rho(n)n$.

Let $\eta>0$ and $n$ and $\Tilde{n}$ be in $\mathcal{A}^0$ such that $\max\left(\|n\|_\mathcal{X},\|\tilde{n}\|_\mathcal{X}\right)<\eta$. We compute
\begin{align*}
    \left\|\hat{f}(n) - \hat{f}(\tilde{n})\right\|_\mathcal{X} &= \left\|\rho(n) n - \rho(\tilde{n})\tilde{n}\right\|_\mathcal{X}\\
    &\leq \left\|\rho(n) (n - \tilde{n})\right\|_\mathcal{X} + \left\|\tilde{n} (\rho(n) - \rho(\tilde{n}))\right\|_\mathcal{X}\\
     &\leq  \rho(n) \left\|n - \tilde{n}\right\|_\mathcal{X} + \left\|\tilde{n}\right\|_\mathcal{X} \left\|\rho(n) - \rho(\tilde{n})\right\|_\mathcal{X}\\
    &\leq \left(\rho_0^+ + \eta\right)\left\|n - \tilde{n}\right\|_\mathcal{X},
    \end{align*}
    where we used the structure of Banach algebra of $(\mathcal{X},\|\cdot\|_\mathcal{X})$ at the third line. Similarly, we compute
\begin{align*}
    \left\|\tilde{f}(n) - \tilde{f}(\tilde{n})\right\|_\mathcal{X} &=\left\|\frac{\rho^X(n)\, \rho^Y(n)}{\rho(n)} - \frac{\rho^X(\tilde{n}) \, \rho^Y(\tilde{n})}{\rho(\tilde{n})}\right\|_\mathcal{X}\\
    &\leq \left\|\left(\rho^X(n)-\rho^X(\tilde{n})\right)\frac{\rho^Y(n)}{\rho(n)}\right\|_\mathcal{X} +  \left\|\rho^X(\tilde{n})\left(\frac{\rho^Y(n)}{\rho(n)}-\frac{\rho^Y(\tilde{n})}{\rho(\tilde{n})}\right)\right\|_\mathcal{X}\\
    &\leq \left\|\rho^X(n)-\rho^X(\tilde{n})\right\|_\mathcal{X}\left\|\frac{\rho^Y(n)}{\rho(n)}\right\|_\mathcal{X} +  \left\|\rho^X(\tilde{n})\right\|_\mathcal{X}\left\|\frac{\rho^Y(n)}{\rho(n)}-\frac{\rho^Y(\tilde{n})}{\rho(\tilde{n})}\right\|_\mathcal{X},\\
    &\leq \frac{|I||J|\eta}{\rho_0^-}\left\|n-\tilde{n}\right\|_\mathcal{X}  \\
    &+\left\|\rho^X(\tilde{n})\right\|_\mathcal{X}\left(\frac{\left\|\rho^Y(n)-\rho^Y(\tilde{n})\right\|_\mathcal{X}}{\rho(n)} + \frac{\left\|\rho^Y(\tilde{n})\right\|_\mathcal{X}}{\rho(n)\rho(\tilde{n})}\left\|\rho(n)-\rho(\tilde{n})\right\|_\mathcal{X}\right),\\
    &\leq \left( 2\frac{|I||J|\eta}{\rho_0^-} +\left[\frac{|I||J|\eta}{\rho_0^-}\right]^2 \right)\left\|n-\tilde{n}\right\|_\mathcal{X},
\end{align*}
where we used the structure of Banach algebra of $(\mathcal{X},\|\cdot\|_\mathcal{X})$ at the third and fourth line. Consequently, we obtain that $f$ is locally Lipschitz on $\mathcal{A}^0$.
\end{proof}}

\section*{Acknowledgments}

Both authors thank Vincent Calvez for the introduction of the model and fruitful discussions and Denis Roze for valuable discussions about the biological motivations.
L.D. thanks also Sarah Otto for enlightening initial conversations.   Both authors have received partial funding from the
ANR project DEEV ANR-20-CE40-0011-01 and the chaire Modélisation Mathématique et Biodiversité of
Véolia Environment - École Polytechnique - Museum National d'Histoire Naturelle - Fondation X. L.D. has received funding through a Mitacs Globalink Research Award.
\printbibliography
\end{document}